\documentclass[11pt]{amsart}
\usepackage[papersize={8.5in,11in},margin=1in]{geometry}

\usepackage{amssymb}
\usepackage{mathtools} 
\usepackage{hyperref}
\usepackage[capitalize]{cleveref}

\usepackage{tikz}
\usepackage{tkz-euclide}

\usepackage[font=small,labelfont=bf]{caption}

\newcounter{copynumber}
\newtheorem{thm}[equation]{Theorem}
\newtheorem{thmcopy}[copynumber]{Theorem}
\newtheorem{prop}[equation]{Proposition}
\newtheorem{propcopy}[copynumber]{Proposition}
\newtheorem{lem}[equation]{Lemma}
\newtheorem{cor}[equation]{Corollary}
\newtheorem{corcopy}[copynumber]{Corollary}
\newtheorem{conj}[equation]{Conjecture}

\newcommand{\fullcref}[2]{\cref{#1}(\ref{#1-#2})}
\crefformat{thm}{Theorem~#2#1#3}
\crefformat{prop}{Proposition~#2#1#3}
\crefformat{lem}{Lemma~#2#1#3}
\crefformat{cor}{Corollary~#2#1#3}
\crefformat{conj}{Conjecture~#2#1#3}
\crefname{thm}{theorem}{theorems}
\crefname{prop}{proposition}{propositions}
\crefname{lem}{lemma}{lemmas}
\crefname{cor}{corollary}{corollaries}
\crefname{conj}{conjecture}{conjectures}
\Crefname{thm}{Theorem}{Theorems}
\Crefname{prop}{Proposition}{Propositions}
\Crefname{lem}{Lemma}{Lemmas}
\Crefname{cor}{Corollary}{Corollaries}
\Crefname{conj}{Conjecture}{Conjectures}

\theoremstyle{definition}
\newtheorem{definition}[equation]{Definition}
\newtheorem{example}[equation]{Example}
\newtheorem{notation}[equation]{Notation}
\newtheorem{Note}[equation]{Note}

\newtheorem{remark}[equation]{Remark}
\newtheorem{Ques}[equation]{Question}

\newcounter{case}
\newcounter{startcases}
\numberwithin{case}{startcases}
\newenvironment{case}[1][\unskip]{\refstepcounter{case}\bf
	\medskip \noindent Case \thecase\ #1. \it}{\unskip\upshape}
\renewcommand{\thecase}{\Roman{case}}
\crefformat{case}{Case~#2#1#3}

\makeatletter
\newcommand{\noprelistbreak}{\@nobreaktrue\nopagebreak\vskip4pt} 
\makeatother

\newenvironment{copyresult}{\thmrefer}{}
\makeatletter
\newcommand{\thmrefer}[1]{\expandafter\ifx\csname r@#1\endcsname\relax
	\renewcommand{\thecopynumber}{??}%
	\protect\G@refundefinedtrue%
	\nfss@text{\reset@font\bfseries ??}%
	\@latex@warning{Reference `#1' on page \thepage \space undefined}%
	\else \cref@getlabel{#1}{\dtemp} \renewcommand{\thecopynumber}{\dtemp} \fi
}
\makeatother

\numberwithin{equation}{section}

\newcommand{\R}{\mathbb{R}}  
\newcommand{\Z}{\mathbb{Z}}  

\DeclareMathOperator{\Area}{Area}

\newcommand{\cartprod}{\mathbin{\Box}}
\newcommand{\cay}{\mathop{\mathsf{Cay}}}
\newcommand{\cayd}{\mathop{\mathsf{C}\mkern-1mu\overrightarrow{\mathsf{ay}}}}

\begin{document}
	
	\title[Arc-disjoint hamiltonian paths in Cartesian products]{Arc-disjoint hamiltonian paths in \\ Cartesian products of directed cycles}
	
	\author{Iren Darijani}
	\email{i.darijani@mun.ca}
	
	\author{Babak Miraftab}
	\email{bobby.miraftab@uleth.ca}
	
	\author{Dave Witte Morris}
	\email{dmorris@deductivepress.ca}
	
	\address{Department of Mathematics and Computer Science, University of Lethbridge, 4401 University
		Drive, Lethbridge, Alberta, T1K 3M4, Canada}
	
	\begin{abstract}
		We show that if $C_1$ and~$C_2$ are directed cycles (of length at least two), then the Cartesian product $C_1 \cartprod C_2$ has two arc-disjoint hamiltonian paths. (This answers a question asked by J.\,A.\,Gallian in 1985.) The same conclusion also holds for the Cartesian product of any four or more directed cycles (of length at least two), but some cases remain open for the Cartesian product of three directed cycles.
		
		We also discuss the existence of arc-disjoint hamiltonian paths in $2$-generated Cayley digraphs on (finite or infinite) abelian groups.
	\end{abstract}
	
	\maketitle
	
	\tableofcontents
	
	\section{Introduction} \label{IntroSect}
	
	It is easy to see (and well known) that the Cartesian product of any two directed cycles has a hamiltonian path. (See \cref{CartProdDefn} for the definition of the Cartesian product.) In 1985, J.\,A.\,Gallian (personal communication) asked whether there are two arc-disjoint hamiltonian paths. The main result of this paper establishes that the answer is ``yes'':
	
	\begin{copyresult}{main}
		\begin{thmcopy} \label{maincopy}
			If $C_1$ and~$C_2$ are directed cycles \textup(of length $\ge 2$\textup), then the Cartesian product $C_1 \cartprod C_2$ has two arc-disjoint hamiltonian paths.
		\end{thmcopy}
	\end{copyresult}
	
	In fact, if the lengths of the directed cycles are large, then there are \emph{many} pairs of arc-disjoint hamiltonian paths:
	
	\begin{copyresult}{ManyInCxC}
		\begin{propcopy}
			Let $N(m,n)$ be the number of \textup(unordered\textup) pairs $\{P, P'\}$ of arc-disjoint hamiltonian paths in the Cartesian product of a directed cycle of length~$m$ and a directed cycle of length~$n$. If $m$ and~$n$ are sufficiently large, then
			\[ N(m,n) > \frac{m^2 \, n^2}{10} . \]
		\end{propcopy}
	\end{copyresult}
	
	Although the \lcnamecref{main} only considers the Cartesian product of precisely two directed cycles, it implies that arc-disjoint hamiltonian paths exist in the Cartesian product of any larger number of directed cycles, except three:
	
	\begin{copyresult}{HamPathsInC1xC2xC3xC4}
		\begin{corcopy}
			If $C_1,C_2,\ldots,C_r$ are directed cycles \textup(of length $\ge 2$\textup), and $r \geq 4$, then the Cartesian product $C_1 \cartprod C_2 \cartprod \cdots \cartprod C_r$ has two arc-disjoint hamiltonian paths.
		\end{corcopy}
	\end{copyresult}
	
	Thus, $r = 3$ is the only open case of the following \lcnamecref{ConjHamPathsInC1xC2xC3}.
	
	\begin{conj} \label{ConjHamPathsInC1xC2xC3}
		If $C_1,C_2,\ldots,C_r$ are directed cycles \textup(of length $\ge 2$\textup), and $r \geq 2$, then the Cartesian product $C_1 \cartprod C_2 \cartprod \cdots \cartprod C_r$ has two arc-disjoint hamiltonian paths.
	\end{conj}
	
	Although the \lcnamecref{ConjHamPathsInC1xC2xC3} has not been proved for all Cartesian products of three directed cycles, we know that it is true in most of these cases:
	
	\begin{copyresult}{C1C2C3}
		\begin{propcopy}
			Assume $C_1$, $C_2$, and~$C_3$ are directed cycles \textup(of length~$\ge 2$\textup). If either
			\begin{enumerate}
				\item
				the Cartesian product of two of the directed cycles has a hamiltonian cycle,
				or
				\item
				the lengths of the three directed cycles do not all have the same parity \textup(i.e., if there is a directed cycle of even length and a directed cycle of odd length\textup),
				or
				\item
				at least one of the directed cycles has length~$2$,
			\end{enumerate}
			then $C_1 \cartprod C_2 \cartprod C_3$ has two arc-disjoint hamiltonian paths.
		\end{propcopy}
	\end{copyresult}
	
	\begin{Ques}
		Does there exist a function $f(r) \to \infty$, such that every Cartesian product of~$r$ directed cycles has at least $f(r)$ pairwise arc-disjoint hamiltonian paths?
	\end{Ques}
	
	It is well known that any Cartesian product of two directed cycles is isomorphic to a $2$-generated Cayley digraph on an abelian group. (See \cref{CayDigDefn,CxCIsCay}.) It is natural to expect that \cref{maincopy} can be extended to this setting:
	
	\begin{conj} \label{ArcDisjointCayleyConj}
		If  $\{a,b\}$ is a $2$-element generating set of a finite abelian group~$G$, and $a$ and~$b$ are nontrivial, then $\cayd(G; a,b)$ has two arc-disjoint hamiltonian paths.
	\end{conj}
	
	We have some evidence that this is true:
	
	\begin{prop}
		\Cref{ArcDisjointCayleyConj} is true in all cases where either:
		\begin{enumerate}
			\item $|G : \langle a-b \rangle|$ is even \textup(see \cref{EvenIndex}\textup),
			or
			\item $|G : \langle a \rangle| > 600$
			\textup(see \cref{LargeIndex}\textup).
		\end{enumerate}
	\end{prop}
	
	\begin{remark}
		Computer calculations verified \cref{ArcDisjointCayleyConj} for all cases where $|G| \le 10^4$. (The \textsf{sagemath} source code for these computations is available in the ancillary files directory of this paper on the arxiv.) Although the calculations are not guaranteed to be error-free, they provide additional evidence for the truth of the conjecture.
	\end{remark}
	
	We also solve the analogous problem for $2$-generated Cayley digraphs on infinite abelian groups. The natural analogue of a hamiltonian path in this setting is a \emph{one-way infinite hamiltonian path}, which means a list
	\[ v_0, v_1, v_2, v_3, \ldots , \]
	of the vertices of the digraph, such that there is a directed edge from $v_i$ to~$v_{i+1}$ for every~$i$. However, it is known (and easy to see) that a $2$-generated Cayley digraph on an infinite abelian group never has a one-way infinite hamiltonian path (cf.\ \cite[Thm.~5.1]{DJungreis} and \cite[Thm.~3.1]{jungreis1985infinite}), so it certainly does not have two of them (whether arc-disjoint or not).
	
	On the other hand, the corresponding natural analogue of a hamiltonian cycle is a \emph{two-way infinite hamiltonian path}, which means a doubly-infinite list
	\[ \ldots, v_{-2}, v_{-1}, v_0, v_1, v_2, \ldots  \]
	of the vertices of the digraph, such that there is a directed edge from $v_i$ to~$v_{i+1}$ for every~$i$. It is well known that these can exist, and we determine exactly when there are two of them that are arc-disjoint:
	
	\begin{copyresult}{infiniteAD}
		\begin{propcopy}
			Assume $G$ is an infinite abelian group.
			\noprelistbreak
			\begin{enumerate}
				\item If there exist $a,b \in G$, such that\/ $\cayd(G; a, b)$ has two arc-disjoint two-way infinite hamiltonian paths, then $G$ is isomorphic to either\/ $\Z$ or\/ $\Z \times \Z_m$, for some $m \ge 2$.
				\item For $a,b \in \Z$, the Cayley digraph\/ $\cayd(\Z; a,b)$ has two arc-disjoint two-way infinite hamiltonian paths if and only if $a$ and~$b$ are odd, and
				\[ \text{either\/ $\{a,b\} = \{1, -1\}$ or $a + b = \pm 2$}. \]
				\item For $a,b \in \Z \times \Z_m$, with $m \ge 2$, the Cayley digraph\/ $\cayd(\Z \times \Z_m; a,b)$ has two arc-disjoint two-way infinite hamiltonian paths if and only if either
				\begin{enumerate}
					\item $\{a,b\} = \{ (1, x), (-1, y)\}$, for some $x,y \in \Z_m$, such that $\langle x + y \rangle = \Z_m$,
					or
					\item $m = 2$, $a = (0,1)$, and $b \in \{\pm 1\} \times \Z_2$, perhaps after interchanging $a$ and~$b$.
				\end{enumerate}
			\end{enumerate}
		\end{propcopy}
	\end{copyresult}

	This paper is structured in the following manner.
	\Cref{review} sets the paper in context by providing a brief review of related results. (None of these results are needed elsewhere in the paper.) 
	\Cref{def} sets some notation and recalls known results that will be used in later sections. (These include basic properties of the ``arc-forcing'' subgroup and some work of Curran-Witte \cite{Curran}.) 
	\Cref{C1xC2Sect} studies the Cartesian product of two directed cycles, and 
	\cref{C1xC2xC3Sect} studies the Cartesian product of more than two directed cycles. 
	\Cref{2genFinite} studies $2$-generated Cayley digraphs on finite abelian groups, and
	\cref{2genInfinite} closes the paper with a discussion of $2$-generated Cayley digraphs on infinite abelian groups.


	\section{Brief review of related results} \label{review}
	
	Several researchers have studied hamiltonian properties of the Cartesian product of two directed cycles. For example, the following result presents three different necessary and sufficient conditions for the existence of a hamiltonian cycle. (We use $\vec{\mathsf C}_n$ to denote a directed cycle of length~$n$.)
	
	\begin{thm} \label{HCiff}
		If $m,n \ge 2$, then the following are equivalent:
		\begin{enumerate}
			\item The Cartesian product $\vec{\mathsf C}_m \cartprod \vec{\mathsf C}_n$ has a hamiltonian cycle.
			\item \textup{(Rankin, cf.\ \cite{Rankin})} \label{HCiff-Rankin}
			There exist $a,b \ge 1$, such that 
			\[ \text{$a + b = \gcd(m,n)$ and $\langle (a,b) \rangle = \langle (1,-1) \rangle$,} \]
			where $\langle (x,y) \rangle$ denotes the subgroup generated by $(x,y)$ in the abelian group $\Z_m \times \Z_n$.
			\item \textup{(Trotter-Erd\H{o}s \cite{TrotterErdos})} There exist $a,b \ge 1$, such that 
			\[ \text{$a + b = \gcd(m,n)$ and $\gcd(a,m) = \gcd(b,n) = 1$.} \]
			\item \textup{(Curran \cite[Thm.~4.3]{witte1984survey})} There exist $a,b \ge 1$, such that 
			\[ 
			\pushQED{\qed} \text{$\gcd(a,b) = 1$ and $am + bn = mn$.} \qedhere
			\popQED 
			\]
		\end{enumerate}
	\end{thm}
	
	\begin{remark}
		References to generalizations, such as determining the lengths of all of the directed cycles in $\vec{\mathsf C}_m \cartprod \vec{\mathsf C}_n$, can be found in the bibliography of~\cite{JungreisGeneralized}.
	\end{remark}
	
	The same methods also determine whether there are two arc-disjoint hamiltonian cycles:
	
	\begin{thm}[Keating {\cite[Cor.~2.4]{keating1985multiple}}] \label{DisjointHC}
		$\vec{\mathsf C}_m \cartprod \vec{\mathsf C}_n$ has two arc-disjoint hamiltonian cycles if and only if there exist $a,b \ge 1$, such that 
		\[ 
		\pushQED{\qed}
		\text{$a + b = \gcd(m,n)$ and $\gcd(ab,mn) = 1$.} 
		\qedhere
		\popQED 
		\]
	\end{thm}
	
	\begin{remark} \label{PathAndCycle}
		\Cref{main} settles the existence of two arc-disjoint hamiltonian paths, and \Cref{DisjointHC} settles the existence of two arc-disjoint hamiltonian cycles. Mixing the two, one might ask about the existence of a hamiltonian path that is arc-disjoint from a hamiltonian cycle. However, this is not interesting, because the complement of a hamiltonian cycle in $\vec{\mathsf C}_m \cartprod \vec{\mathsf C}_n$ must be a union of disjoint directed cycles. Therefore, if the complement contains a hamiltonian path, then it is a hamiltonian cycle.
	\end{remark}
	
	The undirected case is easier. The starting point here is the easy observation that the Cartesian product of two (undirected) cycles is hamiltonian. In fact, it was proved by A.\,Kotzig in 1973 that the Cartesian product of two cycles always has two edge-disjoint hamiltonian cycles. In other words, the Cartesian product of two cycles can be decomposed into hamiltonian cycles. This has been generalized to any number of cycles:
	
	\begin{thm}[Aubert-Schneider, cf.\ {\cite[Thm.~A]{AubertSchneider}}]
		If $C_1,C_2,\ldots,C_r$ are undirected cycles \textup(of length at least $3$\textup), then the Cartesian product $C_1 \cartprod C_2 \cartprod \cdots \cartprod C_r$ can be decomposed into edge-disjoint hamiltonian cycles.\qed
	\end{thm}
	
	It is conjectured that this extends to all Cayley graphs on abelian groups:
	
	\begin{conj}[Alspach, cf.\ {\cite[Unsolved Problem~4.5, p.~454]{alspach1985}}]\label{alspach}
		Let $G$ be an abelian group with a symmetric generating set~$S$.
		If the Cayley graph\/ $\cay(G; S)$ has no loops, then it can be decomposed into hamiltonian cycles \textup(plus a $1$-factor if the valency is odd\textup). 
	\end{conj}
	
	This conjecture is known to be true for valency~$4$:
	
	\begin{thm}[Bermond-Favaron-Maheo  \cite{bermond1989hamiltonian}]
		If $S$ is a $4$-element symmetric generating set of a finite abelian group~$G$ \textup(and $0 \notin S$\textup), then $\cay(G; S)$ has two edge-disjoint hamiltonian cycles.\qed
	\end{thm}
	
	\begin{remark}
		\leavevmode
		\begin{enumerate}
			\item It is not known whether all connected Cayley graphs on nonabelian finite groups have hamiltonian cycles (or hamiltonian paths), but it was proved by D.\,Bryant and M.\,Dean \cite{BryantDean} that not all of them can be decomposed into hamiltonian cycles.
			\item The most recent survey on hamiltonian paths and cycles in Cayley graphs and Cayley digraphs seems to be \cite{Lanel}.
		\end{enumerate}
	\end{remark}

	
	\section{Preliminaries on \texorpdfstring{$2$}{2}-generated Cayley digraphs}\label{def}
	
	This section starts with a few basic definitions, and then recalls some essential results from the 1980's on  hamiltonian paths in $2$-generated Cayley digraphs on finite abelian groups. 
	
	We use standard terminology and notation from the theory of graphs and digraphs. (In particular, ``arc'' is synonymous with ``directed edge\rlap.'') All paths in a digraph are assumed to be \emph{directed} paths.
	
	\begin{notation}
		Throughout this section, 
		\[ \text{$\{a,b\}$ is a $2$-element generating set of a finite abelian group~$G$.} \]
		We use $o(g)$ to denote the order of an element~$g$ of~$G$.
	\end{notation}
	
	\begin{definition}[cf.\ {\cite[p.~504]{GallianBook}}] \label{CayDigDefn}
		The \emph{Cayley digraph} $\cayd(G; a,b)$ of $G$ with respect to the generators $a$ and~$b$ is the digraph whose vertex set is~$G$ and which has a directed edge from $v$ to $v + s$ for every $v \in G$ and $s \in \{a,b\}$.
	\end{definition}
	
	\begin{definition} \label{CartProdDefn}
		Recall that the \emph{Cartesian product} $X \cartprod Y$ of two digraphs $X$ and $Y$ is the digraph with vertex set $V(X) \times V(Y)$ where the vertex $(x_1, y_1)$ is joined to $(x_2, y_2)$ by a directed edge if and only if either $x_1 = x_2$ and $(y_1, y_2) \in E(Y)$ or $y_1 = y_2$ and $(x_1, x_2) \in E(X)$.
	\end{definition}
	
	\begin{example}[cf.\ {\cite[proof of Lem.~1]{Klerlein}}] \label{CxCIsCay}
		The Cartesian product $\vec{\mathsf C}_m \cartprod \vec{\mathsf C}_n$ of two directed cycles can be realized as the Cayley digraph $\cayd(\Z_m \times \Z_n; e_1, e_2)$, where $\{e_1,e_2\} = \{(1,0), (0,1)\}$ is the standard generating set of the abelian group $\Z_m \times \Z_n$.
	\end{example}
	
	\begin{definition}
		\leavevmode
		\noprelistbreak
		\begin{enumerate}
			\item We use the usual terminology that, for $s \in \{a,b\}$, a directed edge of $\cayd(G; a,b)$ is called an \emph{$s$-edge} if it is of the form $(v, v + s)$, for some $v \in G$.
			(So every directed edge is either an $a$-edge or a $b$-edge, but not both.)
			\item
			Suppose $P$ is a subdigraph of $\cayd(G; a,b)$, and let $s \in \{a,b\}$.
			\begin{enumerate}
				\item We let $\delta_s(P)$ be the number of $s$-edges in~$P$.
				\item We say that a vertex~$v$ \emph{travels by~$s$} in~$P$ if $v$ is a vertex of outdegree~$1$ in~$P$, and the out-edge of~$v$ is an $s$-edge \cite[p.~82]{Housman}.
				\item A set of vertices \emph{travels by~$s$} in~$P$ if every element of the set travels by~$s$ (cf.\ \cite[p.~83]{Housman}).
				
			\end{enumerate}
		\end{enumerate}
	\end{definition}

	\subsection{Elementary theory of the arc-forcing subgroup}
	
	\begin{definition}[{\cite[Defn.~2.6]{ForbushEtal}}]
		A \emph{spanning quasi-path} in a digraph~$\Gamma$ is a spanning subdigraph~$P$, such that precisely one connected component of~$P$ is a directed path, and all other components of~$P$ are directed cycles.
	\end{definition}
	
	In particular, every hamiltonian path is a spanning quasi-path.
	Next we provide another characterization.
	
	\begin{remark}[cf.\ {\cite[Defn.~2.6]{ForbushEtal}}]\label{quasi}
		A spanning subdigraph~$P$ of~$\Gamma$ is a spanning quasi-path if and only if there exist vertices $\iota$ and~$\tau$, such that:
		\begin{enumerate}
			\item the indegree of~$\iota$ is~$0$, but the indegree of all other vertices is~$1$,
			and,
			\item the outdegree of~$\tau$ is~$0$, but the outdegree of all other vertices is~$1$.
		\end{enumerate}
		We call $\iota$ the \emph{initial vertex} of~$P$, and call $\tau$~the \emph{terminal vertex} of~$P$. (This acknowledges the fact that they are the initial vertex and terminal vertex of the path component of~$P$.)
	\end{remark}
	
	\begin{definition}
		\leavevmode\noprelistbreak
		\begin{enumerate}
			\item The subgroup $\langle a - b\rangle$ is called the \emph{arc-forcing} subgroup \cite[\S2.3]{witte1984survey}.
			\item Suppose $\tau$ is the terminal vertex of a spanning quasi-path~$P$ of $\cayd(G; a, b)$. 
			Then the coset $\tau + \langle a - b \rangle$ is called the \emph{terminal} coset \cite[Defn.~5(ii)]{morrisNoHP} and all other cosets are called \emph{non-terminal} cosets. (In the older terminology of Housman \cite{Housman}, $\tau + \langle a - b \rangle$ is called the \emph{special} coset and all other cosets are called \emph{regular} cosets.)
		\end{enumerate}
	\end{definition}
	
	Since no vertex of a spanning quasi-path has indegree~$2$, it is clear that if a vertex~$v$ travels by~$a$, then $v + (a - b)$ cannot travel by~$b$. (Similarly, if $v$ travels by~$b$, then $v - (a - b)$ cannot travel by~$a$.) This observation (which is originally due to R.\,A.\,Rankin \cite[Lem.~1]{Rankin}, in the setting of hamiltonian cycles, rather than hamiltonian paths) has the following consequence:
	
	\begin{lem}[Housman {cf.~\cite[pp.~82--83]{Housman}}] \label{arcforcing}
		If $P$ is a spanning quasi-path~$P$ in $\cayd(G; a,b)$, then:
		\begin{enumerate}
			\item \label{arcforcing-nonterminal}
			Each non-terminal coset either travels by~$a$ or travels by~$b$ .
			\item \label{arcforcing-terminal}
			If $d$ is the number of vertices in the terminal coset that travel by~$b$, then:
			\begin{enumerate}
				\item \label{arcforcing-terminal-iota}
				$\iota = \tau + a + d(a - b) \in \tau + a + \langle a - b \rangle$, and
				\item if $v = \tau + i(a - b)$ is any vertex of the terminal coset \textup(with $0 \le i < o(a - b)$\textup), then:
				\begin{itemize}
					\item $v$ travels by~$b$ iff $1 \le i \le d$, and
					\item $v$ travels by~$a$ iff $i > d$.\qed
				\end{itemize}
			\end{enumerate}
		\end{enumerate}
	\end{lem}
	
	\begin{remark}[Housman, cf.\ {\cite[proof of Thm.~1]{Housman}}]
		This implies that a spanning quasi-path in $\cayd(G;a,b)$ is uniquely determined by specifying:
		\begin{enumerate}
			\item the terminal vertex,
			\item how many vertices of the terminal coset travel by~$b$,
			and
			\item the non-terminal cosets that travel by~$b$.
		\end{enumerate}
	\end{remark}
	
	Determining whether a spanning quasi-path is a hamiltonian path does not depend on knowing which particular non-terminal cosets travel by~$b$, or which particular vertices of the terminal coset travel by~$b$, but only on the \emph{number} of each:
	
	\begin{prop}[Housman, cf.~{\cite[Prop.~6.7]{Curran}}] \label{SameNumCosetsAndD}
		Suppose $P$ and~$P'$ are spanning quasi-paths in $\cayd(G; a,b)$, such that 
		\noprelistbreak
		\begin{itemize}
			\item the number of non-terminal cosets that travel by~$b$ in~$P$ is equal to the number of non-terminal cosets that travel by~$b$ in~$P'$,
			and
			\item the number of vertices in the terminal coset that travel by~$b$ in~$P$ is equal to number of vertices in the terminal coset that travel by~$b$ in~$P'$.
		\end{itemize}
		If $P$ is a hamiltonian path, then $P'$ is also a hamiltonian path.\qed
	\end{prop}
	
	This can be restated in terms of the total number of vertices that travel by~$b$ in all of~$P$:
	
	\begin{cor} \label{samedelta}
		Assume that $P$ and~$P'$ are spanning quasi-paths in $\cayd(G; a,b)$, such that $\delta_b(P) = \delta_b(P')$.
		If $P$ is a hamiltonian path, then $P'$ is also a hamiltonian path.
	\end{cor}
	
	\begin{proof}
		Let $t$ (resp.~$t'$) be the number of non-terminal cosets that travel by~$b$ in~$P$ (resp.~in~$P'$), and let $d$ (resp.~$d'$) be the number of vertices of the terminal coset that travel by~$b$ in~$P$ (resp.~in~$P'$). Then $0 \le d,d' < o(a - b)$, and we have
		\[ \text{$\delta_b(P) = t \, o(a - b) + d$ 
			\quad and \quad 
			$\delta_b(P') = t' \, o(a - b) + d'$} .\]
		Since $\delta_b(P) = \delta_b(P')$ by assumption, we conclude that $t = t'$ and $d = d'$. So \cref{SameNumCosetsAndD} tells us that $P'$ is a hamiltonian path.
	\end{proof}
	
	The following construction provides a standard example of a spanning quasi-path in which a particular number~$t$ of non-terminal cosets travel by~$b$, and a particular number~$d$ of vertices in the terminal coset travel by~$b$.
	
	\begin{notation}[{\cite[Defn.~5.2]{Curran}}] \label{Htd}
		For $0 \le t < |G : \langle a-b \rangle|$ and $0 \le d < o(a - b)$, we let $H_t(d)$ be the spanning subdigraph of $\cayd(G; a,b)$, such that:
		\begin{enumerate}
			\item for $0 \le i < o(a - b)$, the element $-a - i(a - b)$ of the coset $-a + \langle a - b \rangle$:
			\begin{itemize}
				\item travels by $b$ if $i < d$,
				\item has no outedges if $i = d$,
				and
				\item travels by $a$ if $i > d$.
			\end{itemize}
			\item the cosets $\langle a - b \rangle, a + \langle a - b \rangle, \ldots, (t-1)a + \langle a - b \rangle$ travel by~$b$,
			and
			\item all other cosets travel by~$a$.
		\end{enumerate}
		This is a spanning quasi-path whose initial vertex is~$0$, and whose terminal vertex is $-a + d(b - a)$.
		By construction, exactly $t$ non-terminal cosets travel by~$b$, and exactly $d$ vertices in the terminal coset travel by~$b$.
	\end{notation}
	
	Since $t$ and~$d$ are arbitrary in \cref{Htd} (and any non-negative integer can be written in the form $t \, o(a - b) + d$, with $0 \le d < o(a - b)$), it follows that there is a spanning quasi-path with any desired number of $b$-edges (in the allowable range):
	
	\begin{lem} \label{anydelta}
		For $0 \le k < |G|$, there is a spanning quasi-path~$P$ in $\cayd(G; a,b)$, such that $\delta_b(P) = k$.\qed
	\end{lem}

	\subsection{Some results of Curran-Witte}
	
	We now briefly recall a few of the results in \cite{Curran}. One of the two main results of the paper is easy to state:
	
	\begin{thm}[Curran-Witte {\cite[Thm.~9.1]{Curran}}] \label{HamCycInC1xC2xC3}
		If $C_1,C_2,\ldots,C_r$ are directed cycles \textup(of length $\ge 2$\textup), and $r \ge 3$, then the Cartesian product $C_1 \cartprod C_2 \cartprod \cdots \cartprod C_r$ has a hamiltonian cycle.\qed
	\end{thm}
	
	The paper's other main result relies on some additional notation. 
	Recall that a point $(x,y)$ in~$\R^2$ is a \emph{lattice point} if its coordinates are integers (i.e., if $x,y \in \Z$), and that a lattice point $(x,y)$ is \emph{primitive} if $\gcd(x,y) = 1$ (cf.~\cite[Prop.~7.2, p.~36]{Clark}).
	
	\begin{notation} \label{mneDefn}
		Let:
		\noprelistbreak
		\begin{enumerate}
			\item $m = o(a)$.
			\item $n = |G : \langle a \rangle|$, so $|G| = mn$.
			\item $e \in \Z$, such that $nb = ea$ and $0 \le e < m$. (A reader who is interested only in Cartesian products of directed cycles can always assume $e = 0$.)
			\item $T(m,n,e)$ be the closed triangle with vertices $(0,0)$, $(n,0)$, and $(e,m)$.
		\end{enumerate}
	\end{notation}
	
	\begin{notation}[{\cite[Notn.~3.4]{Curran}}] \label{XtDefn}
		For $0\leq t \leq |G : \langle a - b \rangle|$, let 
		\[ X_t = \left(1-\frac{t}{I}\right)(n,0) + \frac{t}{I}(e,m) 
		\quad \text{where $I = |G : \langle a - b \rangle|$}
		. \]
	\end{notation}
	
	\begin{Note}[{\cite[Rem.~3.6]{Curran}}] \label{XtIsLatt}
		The sequence $(X_0, X_1, \ldots, X_I)$ is a list of all the lattice points on the line segment joining $(n,0)$ and $(e,m)$. This implies that 
		\[ |G : \langle a - b \rangle| = \gcd(n-e, m) . \]
	\end{Note}
	
	See \cref{uDefnEg} for a concrete example of the values defined in the following notation.
	
	\begin{notation}[cf.\ {\cite[pp.~37--38]{Curran}}] \label{uDefn}
		For $0\leq t < |G : \langle a - b \rangle|$:
		\begin{enumerate}
			\item Let $T_t$ be the closed triangle with vertices $(0,0)$, $X_t$, and~$X_{t+1}$, so
			$\{\, T_t \mid  0 \le t < |G : \langle a - b \rangle| \,\} $
			is a decomposition of the triangle $T(m,n,e)$ into smaller triangles.
			\item Let $A_t(1), A_t(2), \ldots, A_t(f_t)$ be a list of the primitive lattice points in~$T_t$, in order of increasing slope. (So $f_t$ is the number of primitive lattice points in~$T_t$.)
			\item For $1 \le k \le f_t$, let $h_t(k)$ be the number of lattice points that 
			\begin{itemize}
				\item are a positive scalar multiple of $A_t(k)$, and
				\item are in the closed triangle~$T_t$, but
				\item are \emph{not} on the side of~$T_t$ that is opposite $(0,0)$.
			\end{itemize}
			More concretely, we have
			\[ \text{$\displaystyle h_t(k) = \left\lfloor\frac{mn-1}{m \, x_k + (n - e) \, y_k}\right\rfloor$ \quad where $A_t(k) = (x_k,y_k)$} . \]
			\item \label{uDefn-utk}
			For $1 \le k < f_t$, let 
			\[ u_t(k) = h_t(1) + 2 \sum_{j=2}^k h_t(j) . \]
		\end{enumerate}
		This means that $u_t(1) = h_t(1)$, and $u_t(k) = u_t(k - 1) + 2 h_t(k)$ for $2 \le k < f_t$.
		Also note that $u_0(1) = h_0(1) = n-1$. 
	\end{notation}

	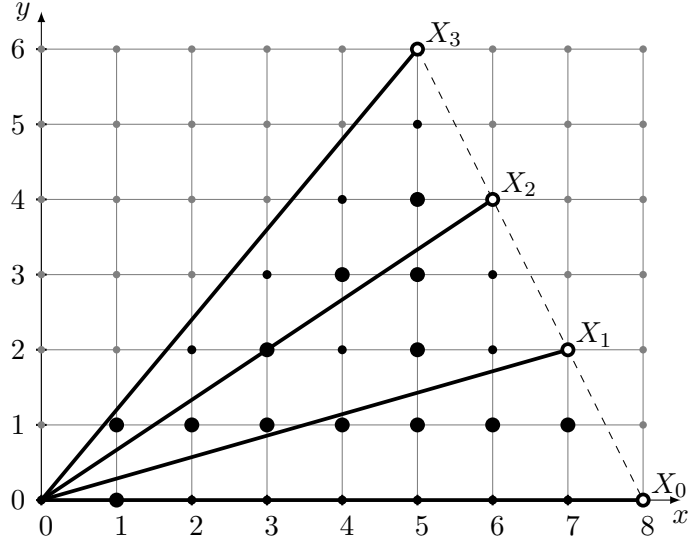
\begin{figure}
		\centering
		\begin{tikzpicture}[every edge/.style = {draw,very thick},
			vertex/.style args = {#1 #2}{,
				minimum size=5mm, 
				label=#1:#2}]
			\tkzInit[xmax=8,ymax=6,xmin=0,ymin=0]
			\tkzGrid
			\tkzAxeXY
			
			\draw[line width=0.5mm] (0,0) -- (8,0) node[anchor=south west] {};
			\draw[line width=0.5mm] (0,0) -- (7,2) node[anchor=south west] {};
			\draw[line width=0.5mm] (0,0) -- (6,4) node[anchor=south west] {};
			\draw[line width=0.5mm] (0,0) -- (5,6) node[anchor=south west] {};
			\draw[ dashed] (5,6) -- (8,0) node[anchor=south west] {}; 
			\draw (8.35,-0.35) node [vertex=above $X_0$] {};
			\draw (7.35,1.65) node [vertex=above $X_1$] {}; 
			\draw (6.35,3.65) node [vertex=above $X_2$] {};
			\draw (5.35,5.65) node [vertex=above $X_3$] {}; 
			
			\foreach \x in {0,...,8}
			\foreach \y in {0,...,6}
			{\draw (\x,\y) node [circle,fill, gray, inner sep=1pt] {}; 
			}
			
			\draw (5,6) node [circle,fill, inner sep=2pt] {}; 
			\draw (6,4) node [circle,fill, inner sep=2pt] {}; 
			\draw (7,2) node [circle,fill, inner sep=2pt] {}; 
			\draw (8,0) node [circle,fill, inner sep=2pt] {}; 
			\draw (7,0) node [circle,fill, inner sep=1.2pt] {}; 
			
			\draw (6,0) node [circle,fill, inner sep=1.2pt] {}; 
			\draw (3,1) node [circle,fill, inner sep=2pt] {}; 
			\draw (2,2) node [circle,fill, inner sep=1.2pt] {}; 
			\draw (3,3) node [circle,fill, inner sep=1.2pt] {}; 
			\draw (1,1) node [circle,fill, inner sep=2pt] {}; 
			\draw (5,3) node [circle,fill, inner sep=2pt] {}; 
			
			\draw (0,0) node [circle,fill, inner sep=1.2pt] {}; 
			\draw (1,0) node [circle,fill, inner sep=2pt] {}; 
			\foreach \x in {2,...,5}
			{\draw (\x,0) node [circle,fill, inner sep=1.2pt] {}; 
			}
			
			\draw (4,1) node [circle,fill, inner sep=2pt] {}; 
			\draw (5,1) node [circle,fill, inner sep=2pt] {}; 
			\draw (6,1) node [circle,fill, inner sep=2pt] {}; 
			\draw (7,1) node [circle,fill, inner sep=2pt] {}; 
			\draw (3,2) node [circle,fill, inner sep=2pt] {}; 
			\draw (2,1) node [circle,fill, inner sep=2pt] {}; 
			
			\draw (4,2) node [circle,fill, inner sep=1.2pt] {}; 
			\draw (5,2) node [circle,fill, inner sep=2pt] {}; 
			\draw (6,2) node [circle,fill, inner sep=1.2pt] {}; 
			\draw (6,3) node [circle,fill, inner sep=1.2pt] {}; 
			\draw (4,4) node [circle,fill, inner sep=1.2pt] {}; 
			\draw (5,5) node [circle,fill, inner sep=1.2pt] {}; 
			
			\draw (4,3) node [circle,fill, inner sep=2pt] {}; 
			\draw (5,4) node [circle,fill, inner sep=2pt] {}; 
			
			\draw (5,6) node [circle,fill,white, inner sep=1pt] {}; 
			\draw (6,4) node [circle,fill,white, inner sep=1pt] {}; 
			\draw (7,2) node [circle,fill,white, inner sep=1pt] {}; 
			\draw (8,0) node [circle,fill,white, inner sep=1pt] {}; 
			
		\end{tikzpicture}
		\captionof{figure}{Primitive lattice points in the triangle $T(6,8,5)$.}
		\label{Prim685}
	\end{figure}

	\begin{figure}
		$\begin{array}{|c|c|c|c|c|r|}
			\hline
			t & f_t & k & A_t(k) & h_t(k) & u_t(k) \hphantom{= 09} \\
			\hline
			0 & 6 & 1 & (1,0) & 7 & 7 \\
			& & 2 & (7,1) & 1 & 7 + 2 \times 1 = \phantom0 9 \\
			& & 3 & (6,1) & 1 & 9 + 2 \times 1 = 11 \\
			& & 4 & (5,1) & 1 & 11 + 2 \times 1 = 13 \\
			& & 5 & (4,1) & 1 & 13 + 2 \times 1 = 15 \\
			& & 6 & (7,2) & 0 &  \\
			\hline
			1 & 6 & 1 & (7,2) & 0 & 0 \\
			& & 2 & (3,1) & 2 & 0 + 2 \times 2 = \phantom0 4 \\
			& & 3 & (5,2) & 1 & 4 + 2 \times 1 = \phantom0 6 \\
			& & 4 & (2,1) & 3 & 6 + 2 \times 3 = 12 \\
			& & 5 & (5,3) & 1 & 12 + 2 \times 1 = 14 \\
			& & 6 & (3,2) & 1 &  \\
			\hline
			2 & 5 & 1 & (3,2) & 1 & 1 \\
			& & 2 & (4,3) & 1 & 1 + 2 \times 1 = \phantom0 3 \\
			& & 3 & (5,4) & 1 & 3 + 2 \times 1 = \phantom0 5 \\
			& & 4 & (1,1) & 5 & 5 + 2 \times 5 = 15 \\
			& & 5 & (5,6) & 0 &  \\
			\hline
		\end{array}$
		\caption{The functions $A_t$, $h_t$, and~$u_t$ in the case where $(m,n,e) = (6,8,5)$. The primitive lattice points $A_t(k)$ are taken from \cref{Prim685}.}
		\label{uDefnEg}
	\end{figure}

	\begin{remark}
		For each~$t$, Curran and Witte \cite[Notation 3.14]{Curran} also defined a certain function
		\[ B_t \colon \{ 0, \, 1, \, 2, \, \ldots, \, o(a - b) - 1 \} \to \Z \times \Z , \]
		such that
		\begin{align} \label{B_t(d)=0} \tag{$*$}
			\text{$B_t(d) = (0,0)$ \quad$\Leftrightarrow$\quad $d = u_t(k)$ for some~$k$ (with $1 \le k < f_t$).}
		\end{align}
		We will not define the function $B_t$, because we have no need for it, other than its appearance in the statement of the following \lcnamecref{HPiff}, which is the other main result of~\cite{Curran}. (Actually, the statement of \cite[Thm.~7.1]{Curran} only includes ($\ref{HPiff-Htd} \Leftrightarrow \ref{HPiff-Btd}$) of the following result, but ($\ref{HPiff-Btd} \Leftrightarrow \ref{HPiff-utk}$) is the fundamental property of $B_t(d)$ that is stated in~(\ref{B_t(d)=0}) above. All we need is the equivalence ($\ref{HPiff-Htd} \Leftrightarrow \ref{HPiff-utk}$).)
	\end{remark}
	
	\begin{thm}[Curran-Witte {\cite[Thm.~7.1]{Curran}}] \label{HPiff}
		For $0 \le t < |G : \langle a-b \rangle|$ and $0 \le d < o(a - b)$, the following are equivalent:
		\begin{enumerate}
			\item \label{HPiff-Htd}
			$H_t(d)$ is a hamiltonian path.
			\item \label{HPiff-Btd}
			$B_t(d) = (0,0)$.
			\item \label{HPiff-utk}
			$d = u_t(k)$ for some $k \in \{1,2,\ldots, f_t - 1\}$.\qed
		\end{enumerate}
	\end{thm}
	
	\begin{cor}[Curran-Witte {\cite[Rem.~8.4]{Curran}}] \label{NumHP}
		\[ \# \{ \, (t,d) \mid \text{$H_t(d)$ is a hamiltonian path in $\cayd(G; a,b)$} \,\} \]
		is exactly one less than the number of primitive lattice points in the closed triangle $T(m,n,e)$.\qed
	\end{cor}
	
	\begin{lem}[{\cite[Prop.~3.13]{Curran}}]\label{vertices}
		For $0 \le t < |G \colon \langle a - b \rangle|$, we have
		\[ 
		\pushQED{\qed} 
		u_t(f_t - 1) + h_t(f_t) = o(a - b) - 1. 
		\qedhere
		\popQED
		\]
	\end{lem}
	
	This has the following immediate consequences:
	
	\begin{cor} \label{ut+ut}
		For $0 < t < |G \colon \langle a - b \rangle|$, we have
		\[ u_{t-1}(f_{t-1} - 1) = o(a - b) - 1 - u_t(1) . \]
	\end{cor}
	
	\begin{proof}
		Note that $A_{t-1}(f_{t-1}) = A_t(1)$ (because this is the primitive lattice point on the edge that $T_{t-1}$ and $T_t$ have in common), so we have $h_{t-1}(f_{t-1}) = h_t(1)$. Since $h_t(1) = u_t(1)$ (by the definition of $u_t(k)$), the desired conclusion is now immediate from \cref{vertices} (after replacing~$t$ with~$t-1$).
	\end{proof}
	
	\begin{cor} \label{parity}
		Assume $0 \le t, t' < |G : \langle a - b \rangle|$, $1 \le k < f_t$, and $1 \le k' \le f_{t'}$.
		\begin{enumerate}
			\item \label{parity-odd}
			If $o(a - b)$ is odd, then $u_t(k) \equiv u_{t'}(k') \pmod{2}$.
			\item \label{parity-even}
			If $o(a - b)$ is even, then $u_t(k) - u_{t'}(k') \equiv t - t' \pmod{2}$.
		\end{enumerate}
	\end{cor}
	
	\begin{proof}
		It is immediate from \fullcref{uDefn}{utk} that the parity of~$u_t(k)$ depends only on~$t$. Then we see from \cref{ut+ut} that $u_{t-1}(*)$ has the same parity as $u_t(*)$ if and only if $o(a - b)$ is odd.
	\end{proof}
	
	In the remainder of this section, we will give two consequences of \cref{HPiff}. They are stated in \cite{Curran} only for the special case where $\cayd(G; a,b)$ is the Cartesian product of two directed cycles, but the same arguments apply to the general case. For completeness, we provide proofs.
	
	\begin{prop}[Curran-Witte] \label{string}
		Assume that $m$ and~$n$ are as in \cref{mneDefn}.
		If $m \ge 2$, then there exists $n' \in \{ n - 1, n \}$, such that, for each $i$ in
		\[ \{ n', n' + 2, n' + 4, \ldots, n' + 2 \lfloor (n - 1)/2 \rfloor \} , \]
		there is a hamiltonian path~$P_i$ in $\cayd(G; a,b)$, such that $\delta_b(P_i) = i$.
	\end{prop}
	
	\begin{proof}[Proof \textup{(cf.\ \cite[(9.2) on p.~67, and Case~3 on p.~69]{Curran})}]
		For convenience, we let $T = T(m,n,e)$.
		Let $x_1,x_2 \in \R$, such that $(x_1,1)$ and $(x_2, 2)$ are on the line segment from $(n,0)$ to $(e,m)$. (This line segment is the side of~$T$ that is opposite the vertex $(0,0)$.) Then $(x_1, 1)$ is the midpoint of the line segment from $(n,0)$ to $(x_2, 2)$, so 
		\[ x_1 = \frac{x_2}{2} + \frac{n}{2} . \]
		
		We consider two cases, but the argument is almost the same for both.
		
		\refstepcounter{startcases}
		\begin{case} \label{stringPf-x1notinZ}
			Assume $x_1 \notin \Z$.
		\end{case}
		For convenience, let $x_1^* = \lfloor x_1 \rfloor$, and let 
		\[ I = \{\, x_1^*, x_1^* - 1, x_1^* - 2, \ldots, x_1^* - \bigl( \lfloor n/2 \rfloor - 1 \bigr) \,\} . \]
		For all $x \in I$, it is obvious that $x \le x_1^* < x_1$. Also, we have 
		\[ x 
		\ge x_1^* - \bigl( \lfloor n/2 \rfloor - 1 \bigr)
		> (x_1 - 1) - \bigl( (n/2) - 1 \bigr)
		= \frac{x_2}{2}
		.\]
		So the set $\{\, (x, 1) \mid x \in I\,\}$ is contained in the interior of~$T$, and the slope of each point in this set is strictly smaller than the slope of any nonzero lattice point in~$T$ that is not in this set, other than the points $(1,0), (2,0), \ldots, (n,0)$ that are on the $x$-axis (and therefore have slope~$0$). 
		This implies that $(x,1)$ is in~$T_0$ for all $x \in I$. More precisely, since it is obvious that the lattice point $(x,1)$ is primitive, we have
		\[ \{\, A_0(k) \mid 2 \le k \le \lfloor n/2 \rfloor + 1 \,\} = \{\, (x, 1) \mid x \in I\,\} . \]
		Since $2x > x_2$ for all $x \in I$, this implies $2A_0(k) \notin T$, so
		\[ \text{$h_0(k) = 1$ for $2 \le k \le \lfloor n/2 \rfloor + 1$} . \]
		By the definition of $u_t(k)$, we then have 
		\[ u_0(k) - u_0(k-1) = 2 h_0(k) = 2 \]
		for these values of~$k$, so
		\begin{align} \label{string-u0}
			\{\, u_0(k) & \mid 1 \le k \le \lfloor n/2 \rfloor + 1 \,\}
			\\&= \{ n-1, n+1, n + 3\ldots, n - 1 + 2 \lfloor n/2 \rfloor \}
			\notag
		\end{align}
		Hence, \cref{HPiff} tells us that $H_0(d)$ is a hamiltonian path for all $d$ in this set.
		Since 
		\[ \delta_b \bigl( H_0(d) \bigr)
		= 0 \cdot o(a - b) + d
		= d ,\]
		and $2 \lfloor n/2 \rfloor \ge 2 \lfloor (n - 1)/2 \rfloor$, this implies that we may let $n' = n-1$ and $P_i = H_0(i)$.
		
		\begin{case}
			Assume $x_1 \in \Z$.
		\end{case}
		This means that $X_1 = (x_1,1)$.
		Let
		\[ I = \{\, x_1, x_1 - 1, x_1 - 2, \ldots, x_1 - \lfloor (n-1)/2 \rfloor \,\} . \]
		For all $x \in I$, we have 
		\[ x \ge x_1 - \lfloor (n-1)/2 \rfloor > x_1 - (n/2) = \frac{x_2}{2} , \]
		so, as in the previous case, the points in $\{\, (x, 1) \mid x \in I \,\}$ are in~$T$, and have strictly smaller slope than any other lattice points in~$T$, other than the lattice points on the $x$-axis. Since $(x_1,1) = X_1$, this implies $(x,1)$ is in~$T_1$. More precisely, we have
		\[ \{\, A_1(k) \mid 1 \le k \le \lfloor (n-1)/2 \rfloor + 1 \,\} = \{\, (x,1) \mid x \in I \,\} . \]
		Since $(x,1)$ is primitive (and $(x_1,1)$ is on the boundary of~$T$, not in the interior), this implies that
		\[ \displaystyle h_0(k) = 
		\begin{cases}
			0 & \text{if $k = 1$} \\
			1 & \text{otherwise}
		\end{cases}
		\quad 
		\text{for $1 \le k \le \lfloor (n-1)/2 \rfloor + 1$}
		. \]
		Therefore $H_1(d)$ is a hamiltonian path for all $d$ in
		\[ \{\, u_1(k) \mid 1 \le k \le \lfloor (n-1)/2 \rfloor + 1 \,\} = \{ 0, 2, 4, \ldots, 2 \lfloor (n-1)/2 \rfloor \} . \]
		Note that, since the $y$-coordinate of $X_1 = (x_1, 1)$ is~$1$, it follows from the definition in \cref{XtDefn} that $m = |G : \langle a - b \rangle|$, so
		\[ o(a - b) 
		= \frac{|G|}{|G : \langle a - b \rangle|}
		= \frac{mn}{m}
		= n . \]
		Therefore
		\[ \delta_b \bigl( H_1(d) \bigr)
		= 1 \cdot o(a - b) + d
		= 1 \cdot n + d
		= n + d .\]
		Therefore, we may let $n' = n$ and $P_i = H_1(i)$.
	\end{proof}
	
	\begin{prop}[Curran-Witte, cf.\ {\cite[(9.5) on p.~68]{Curran}}] \label{dgap}
		Assume that $m$, $n$, and~$e$ are as in \cref{mneDefn}. Also assume $m \ge 2$ and $n \ge e$.
		
		For every integer~$k$, such that $n - 1 \le k \le m(n-1)$, there is a hamiltonian path~$P$ in $\cayd(G; a,b)$, such that
		\[ k \le \delta_b(P) \le k + 2 \left\lfloor \frac{mn - 1}{m + n - e} \right\rfloor . \]
		Therefore, if $n \ge m + e$, then 
		\[ k \le \delta_b(P) \le k + 2 \left\lfloor \frac{n - 1}{2} \right\rfloor. \]
	\end{prop}
	
	\begin{proof}
		First, note that if $n \ge m + e$, then 
		\[ \frac{mn - 1}{m + n - e}
		\le \frac{mn - 1}{2m} 
		< \frac{mn}{2m} 
		= \frac{n}{2} , \]
		so
		\[ \left\lfloor \frac{mn - 1}{m + n - e} \right\rfloor
		\le \left\lfloor \frac{n - 1}{2} \right\rfloor
		. \]
		Therefore, the final sentence of the proposition follows from the one that precedes it.
		
		For convenience, let $N = o(a - b)$ and $M = |G|/N - 1$. In this notation, we have
		\begin{align*}
			\{\, \delta_{b}(P) \mid & \text{$P$ is a hamiltonian path} \,\}
			\\&= \{ u_0(1), u_0(2), \ldots, u_0(f_0-1),
			\\& \qquad N + u_1(1), N + u_1(2), \ldots, N + u_1(f_1-1),
			\\& \qquad 2N + u_2(1), 2N + u_2(2), \ldots, 2N + u_2(f_2-1),
			\\& \qquad \qquad \qquad  \vdots 
			\\& \qquad MN + u_M(1), MN + u_M(2), \ldots, MN + u_M(f_M-1) \}
			. \end{align*}
		Since $u_0(1) = n - 1$ and
		\begin{align*}
			MN + u_M(f_M - 1) 
			&= \bigl(|G| - N \bigr) + \bigl( N - 1 - h_M(f) \bigr)
			\\&= mn - 1 - h_M(f)
			\ge mn - m
			= m(n-1)
			, \end{align*}
		the assertion we wish to prove is that the difference between successive terms in this list is never larger than $2 \lfloor (mn - 1)/(m + n - e) \rfloor + 1$. It therefore suffices to show 
		\begin{align} \label{dgapPf-ugap}
			\text{$u_t(k) - u_t(k-1) \le 2 \left\lfloor \frac{mn - 1}{m + n - e} \right\rfloor + 1$ for $2 \le k < f_t$}
		\end{align}
		and 
		\begin{align} \label{dgapPf-intergap}
			\text{$N + u_{t+1}(1) - u_t(f_t - 1) \le 2 \left\lfloor \frac{mn - 1}{m + n - e} \right\rfloor + 1$ for $0 \le t < M$}
			. \end{align}
		
		To establish~(\ref{dgapPf-ugap}), note that (much as in \cite[Case~1 on p.~67]{Curran}) we have
		\begin{align*}
			u_t(k) - u_t(k-1) 
			&= 2 h_k 
			\\&= 2 \left\lfloor \frac{mn - 1}{m \, x_k + (n - e) y_k} \right\rfloor 
			&& \text{(where $x_k, y_k \in \Z^+$)}
			\\&\le 2 \left\lfloor \frac{mn - 1}{m + n - e} \right\rfloor
			&& \begin{pmatrix} n - e \ge 0 \\ \text{and $x_k, y_k \ge 1$} \end{pmatrix}
			. \end{align*}
		
		To establish~(\ref{dgapPf-intergap}), we make three observations. First, note that, by definition, we have $u_{t+1}(1) = h_{t+1}(1)$. Second, we know from \cref{vertices} that 
		\[ u_t(f_t - 1) = N - 1 - h_t(f) .\]
		Third, we have $A_t(f_t) = A_{t+1}(1)$. (This is the primitive lattice point on the boundary between $T_t$ and~$T_{t + 1}$.) So
		\[ h_{t+1}(1) 
		= h_t(f) 
		= \left\lfloor \frac{mn - 1}{m x_f + (n - e) y_f} \right\rfloor
		\le \left\lfloor \frac{mn - 1}{m + n - e} \right\rfloor
		. \]
		Therefore
		\begin{align*} N &+ u_{t+1}(1) - u_t(f_t - 1)
			\\&= N + h_t(f) - \bigl( N - 1 - h_t(f) \bigr)
			\\&= 2 h_t(f) + 1
			\\&\le 2 \left\lfloor \frac{mn - 1}{m + n - e} \right\rfloor + 1
			. \qedhere \end{align*}
	\end{proof}
	
	To provide a convenient reference, we now restate \cref{dgap} for the special case where $\cayd(G; a,b)$ is the Cartesian product of two directed cycles:
	
	\begin{cor} \label{dgapCartProd}
		Let $\{e_1, e_2\}$ be the standard generating set of $\Z_m \times \Z_n$, and assume $2 \le m \le n$. If $1 \le k \le m(n - 1)$, then there is a hamiltonian path~$P$ in $\cayd(\Z_m \times \Z_n; e_1, e_2)$, such that
		\[ k \le \delta_{e_2}(P) \le k + 2 \left\lfloor \frac{n - 1}{2} \right\rfloor . \]
	\end{cor}
	
	\begin{proof}
		We have $e = 0$, so the assumption that $m \le n$ implies $n \ge m + e$.  Therefore, if $k \ge n-1$, then the desired conclusion is obtained from the final sentence in the statement of \cref{dgap}.
		
		So we may now assume $k < n-1$. Then, since 
		\[ n - 1 - k \le n - 2 \le 2 \lfloor (n-1)/2 \rfloor , \]
		we may let $P = H_0(n-1)$. (Since $u_0(1) = n - 1$, we know that $H_0(n-1)$ is a hamiltonian path. We also have $\delta_{e_2} \bigl( H_0(n-1) \bigr) = n-1$.)
	\end{proof}


	\section{Cartesian product of two directed cycles} \label{C1xC2Sect}
	
	In this section, we show that the Cartesian product of two directed cycles has two arc-disjoint hamiltonian paths. We start with the following lemma.
	
	\begin{lem} \label{makedisjoint}
		Let $\{a,b\}$ be a\/ $2$-element generating set of a finite abelian group~$G$, and let $k$ and~$\ell'$ be integers such that 
		\[ \text{$0 \le k,\ell' < |G|$ and $|k - \ell'| \le 1$.} \] 
		Then there exist two arc-disjoint spanning quasi-paths $P$ and $P'$ in $\cayd(G; a,b)$, such that $\delta_b (P) = k$ and $ \delta_a(P') = \ell'$.
	\end{lem}
	
	\begin{proof}
		By \cref{anydelta}, there is a spanning quasi-path~$P$, such that $\delta_b(P) = k$. Let $\iota$ and~$\tau$ be the initial vertex and terminal vertex of~$P$, respectively.
		
		Let $\overline P$ be the complement of~$P$ in $\cayd(G; a,b)$. (Thus, a subdigraph of $\cayd(G; a,b)$ is arc-disjoint from~$P$ if and only if it is contained in~$\overline P$.) The vertices that have some $a$-edge as an out-edge in~$\overline P$ are precisely the vertices that do not travel by~$a$ in~$P$; this consists of the vertices that travel by~$b$ in~$P$, plus the terminal vertex~$\tau$. Hence, we have
		\[ \delta_a(\overline P) = \delta_b(P) + 1 = k + 1. \]
		
		Note that $\overline P$ is \emph{not} a spanning quasi-path, because the indegree of~$\iota$ is~$2$, and the outdegree of~$\tau$ is~$2$. (The indegree and outdegree of each of the other vertices is~$1$.)
		
		\refstepcounter{startcases}
		\begin{case}
			Assume $\overline P$ does not have a directed edge from~$\tau$ to~$\iota$.
		\end{case}
		If we construct a subdigraph~$P'$ of~$\overline P$ by removing either of the two in-edges of~$\iota$ and either of the two out-edges of~$\tau$, then $P'$ will be a spanning quasi-path. (The initial vertex of the removed in-edge will have outdegree~$0$, and the terminal vertex of the removed out-edge will have indegree~$0$.) We can freely decide whether neither, precisely one, or both of the directed edges we remove is an $a$-edge, so we can construct $P'$ so that $\delta_a(P')$ is any desired element of 
		\[ \{\delta_a(\overline P), \delta_a(\overline P) - 1, \delta_a(\overline P) - 2\}
		= \{ k + 1, k, k - 1 \}. \]
		Since $|k - \ell'| \le 1$, this means we can construct $P'$ so that $\delta_a(P')$ is equal to~$\ell'$, as desired.
		
		\begin{case}
			Assume $(\tau, \iota)$ is a directed edge of~$\overline P$.
		\end{case}
		Assume, without loss of generality, that $(\tau,\iota)$ is a $b$-edge (by interchanging $a$ and~$b$ if necessary). 
		
		If $\ell' = k - 1$, then $\ell' = \delta_a(\overline P) - 2$, so we construct $P'$ by removing two $a$-edges from~$\overline{P}$: the out-edge of~$\tau$ that is labelled~$a$, and the in-edge of~$\iota$ that is labelled~$a$. (These are two different directed edges, because we have assumed that the directed edge $(\tau,\iota)$ is a $b$-edge.)
		
		We may now assume that 
		\[ \ell' \in \{ k + 1, k \}
		= \{\delta_a(\overline P), \delta_a(\overline P) - 1\} . \]
		Removing the $b$-edge $(\tau, \iota)$ from $\overline P$ results in a spanning subdigraph~$P^*$ in which the indegree and outdegree of every vertex is~$1$, so each connected component of~$P^*$ is a directed cycle. Hence, removing any directed edge from~$P^*$ will result in a spanning quasi-path~$P'$. Thus, if $P^*$ has both an $a$-edge and a $b$-edge, then we can choose which type of directed edge to remove, so we can arrange for $\delta_a(P')$ to be either element of 
		\[ \{\delta_a(\overline P), \delta_a(\overline P) - 1\} , \]
		so we can certainly arrange for $\delta_a(P')$ to be equal to~$\ell'$.
		
		Since $P^*$ does have $a$-edges (such as the out-edge of~$\tau$), we may now assume that every directed edge of~$P^*$ is an $a$-edge. This means that $\delta_a(P^*) = |G|$, so 
		\[ k = \delta_a(\overline{P}) - 1 
		= \delta_a(P^*) - 1 = |G| - 1 . \]
		Since $|k - \ell'| \le 1$ and $\ell' < |G|$, this implies that $\ell' \in \{ |G| - 1, |G| - 2 \}$. If $\ell' = |G| - 1$, we can construct the desired spanning quasi-path~$P'$ by removing one of the $a$-edges of~$P^*$. If $\ell' = |G| - 2$, we can construct $P'$ by removing two $a$-edges from~$\overline{P}$: an out-edge of~$\tau$ and an in-edge of~$\iota$.
	\end{proof}
	
	\begin{remark} \label{MustBe<1}
		It is easy to see that the converse of \cref{makedisjoint} is true. Namely, if $P$ and~$P'$ are arc-disjoint spanning quasi-paths in $\cayd(G; a,b)$, then we have
		\[ |\delta_b (P) - \delta_a (P')| \le 1 . \]
	\end{remark} 
	
	\begin{prop} \label{iflessthanone}
		Let $\{a,b\}$ be a $2$-element generating set of a finite abelian group~$G$, and assume that $\cayd(G; a,b)$ has hamiltonian paths $P$ and~$P'$ that satisfy either of the following two equivalent conditions:
		\begin{enumerate}
			\item \label{iflessthanone-ba}
			$|\delta_b(P) - \delta_a(P')| \le 1$,
			or
			\item \label{iflessthanone-bb}
			$|\delta_b(P) + \delta_b(P')| \in \{|G|, |G| - 1, |G| - 2\}$.
		\end{enumerate}
		Then $\cayd(G; a,b)$ has two arc-disjoint hamiltonian paths.
	\end{prop}
	
	\begin{proof}
		First, note that the equivalence of the two conditions is immediate from the observation that $\delta_a(P') = |G| - 1 -  \delta_b(P')$. Therefore, we may assume that~(\ref{iflessthanone-ba}) holds.
		
		\Cref{makedisjoint} tells us that there exist two arc-disjoint spanning quasi-paths $P_0$ and~$P_0'$ in $\cayd(G; a,b)$, such that $\delta_b (P_0) = \delta_b (P)$ and $ \delta_a (P_0') = \delta_a (P')$. Since $\delta_b (P_0) = \delta_b (P)$ and
		\[ \delta_b (P_0') 
		= |G| - 1 - \delta_a (P_0') 
		= |G| - 1 - \delta_a (P')
		= \delta_b (P') , \]
		we see from \cref{samedelta} that $P_0$ and~$P_0'$ are hamiltonian paths.
		Since $P_0$ and~$P_0'$ are arc-disjoint, this completes the proof.
	\end{proof}
	
	We are now ready to prove the main theorem of the paper.
	
	\begin{thm}\label{main}
		If $C_1$ and~$C_2$ are directed cycles \textup(of length $\ge 2$\textup), then the Cartesian product $C_1 \cartprod C_2$ has two arc-disjoint hamiltonian paths.
	\end{thm}
	
	\begin{proof}
		Let $m$ and~$n$ be the lengths of~$C_1$ and~$C_2$, so
		\[ C_1 \cartprod C_2 \cong \cayd(\mathbb{Z}_m \times \mathbb{Z}_n; e_1, e_2) , \] 
		where $\{e_1, e_2\} = \{(1,0), (0,1)\}$ is the standard generating set of the abelian group $\mathbb{Z}_m \times \mathbb{Z}_n$. 
		Assume without loss of generality that $m \le n$. Then \cref{dgapCartProd} provides a hamiltonian path~$P'$, such that
		\[ mn - n - 2\left\lfloor \frac{n-1}{2} \right\rfloor 
		\le \delta_{e_2}(P') \le mn - n . \]
		Since $\delta_{e_1}(P') = mn - 1 - \delta_{e_2}(P')$, we have
		\[ n - 1 \le \delta_{e_1}(P') 
		\le n - 1 + 2\left\lfloor \frac{n-1}{2} \right\rfloor . \]
		Then it is immediate from \cref{string} that there is a hamiltonian path~$P_i$, such that we have $|\delta_{e_2}(P_i) - \delta_{e_1}(P')| \le 1$. The desired conclusion now follows from \fullcref{iflessthanone}{ba}.
	\end{proof}
	
	The above result shows that there is always at least one pair of arc-disjoint hamiltonian paths. In fact, the number of such pairs is large if both directed cycles in the product have large length:
	
	\begin{prop}[see \cref{manyHP} below] \label{ManyInCxC}
		Let $N(m,n)$ be the number of \textup(unordered\textup) pairs $P, P'$ of arc-disjoint hamiltonian paths in the Cartesian product of a directed cycle of length~$m$ and a directed cycle of length~$n$. If $m$ and~$n$ are sufficiently large, then
		\[
		\pushQED{\qed} 
		N(m,n) > \frac{m^2 \, n^2}{10} .
		\qedhere
		\]
	\end{prop}
	
	If $\cayd(G; a,b)$ has two arc-disjoint hamiltonian cycles, then removing a directed edge from each of these hamiltonian cycles yields two arc-disjoint hamiltonian paths $P$ and~$P'$. Note that the initial vertex~$\iota'$ of~$P'$ is completely arbitrary, even after the path~$P$ has been chosen, because there is no restriction on the directed edge that is removed. (The terminal vertex of~$P'$ is also arbitrary, but it is determined by the choice of~$\iota'$.) In all other situations, the following observation shows that $P'$ is almost completely determined by the choice of~$P$.
	
	\begin{prop} \label{iota=tau}
		Suppose $P$ and~$P'$ are two arc-disjoint hamiltonian paths in $\cayd(G; a,b)$. Let $\iota$ and~$\iota'$ be the initial vertices of these hamiltonian paths, and let $\tau$ and~$\tau'$ be the terminal vertices. Assume there is no directed edge from~$\tau$ to~$\iota$ \textup(or, equivalently by \cref{PathAndCycle}, that there is no directed edge from~$\tau'$ to~$\iota'$\textup). Then:
		\begin{enumerate}
			\item \label{iota=tau-iota}
			$\iota' \in \{\tau + a, \tau + b\}$,
			\item \label{iota=tau-tau}
			$\tau' \in \{\iota - a, \iota - b\}$.
			\item \label{iota=tau-same}
			$\iota$ and~$\tau'$ are in the same coset of $\langle a - b \rangle$,
			\item \label{iota=tau-same'}
			$\iota'$ and~$\tau$ are in the same coset of $\langle a - b \rangle$,
			and
			\item \label{iota=tau-unique}
			there are at most two hamiltonian paths that are arc-disjoint from~$P$.
		\end{enumerate}
	\end{prop}
	
	\begin{proof}
		Let $\overline P$ be the complement of~$P$ in $\cayd(G; a,b)$, as in the proof of \cref{makedisjoint}. Note that $P'$ can be obtained from~$\overline P$ by removing one of the two in-edges of~$\iota$ and one of the two out-edges of~$\tau$. (Since there is no directed edge from~$\tau$ to~$\iota$, the in-edges of~$\iota$ are distinct from the out-edges of~$\tau$.) Thus, we have 
		(\ref{iota=tau-iota})~$\iota' \in \{\tau + a, \tau + b\}$
		and
		(\ref{iota=tau-tau})~$\tau' \in \{\iota - a, \iota - b\}$.
		Combining these with \fullcref{arcforcing}{terminal-iota} yields~(\ref{iota=tau-same}) and~(\ref{iota=tau-same'}).
		
		(\ref{iota=tau-unique}) Since there are only two possible choices for the in-edge of~$\iota$ to remove, and two possible choices for the out-edge of~$\tau$ to remove, it is obvious that no more than 4 hamiltonian paths are arc-disjoint from~$P$. 
		
		To complete the proof, we show there cannot be more than one hamiltonian path that starts at $\tau + a$ and is arc-disjoint from~$P$. (By symmetry, there is also no more than one that starts at $\tau + b$). Suppose, for a contradiction, that the terminal vertex of~$P'$ is $\tau' = \iota - a$ and the terminal vertex of some other arc-disjoint hamiltonian path~$P''$ is $\tau'' = \iota - b$. Then $\tau' - \tau'' = b - a$. Thus, if $d$ is the number of vertices in $\tau' + \langle a - b \rangle$ that travel by~$b$ in~$P'$, then the the number of vertices in this coset that travel by~$b$ in~$P''$ is $d + 1$. Since $P'$ and~$P''$ are hamiltonian paths, we conclude from \cref{HPiff} (and \cref{SameNumCosetsAndD}) that $d = u_t(k')$ and $d + 1 = u_t(k'')$, for some $k'$ and~$k''$, where $t$ is the number of non-terminal cosets that travel by~$b$ in~$P'$ (or, equivalently, in~$P''$). However, it is clear from \fullcref{uDefn}{utk} that two values of $u_t(*)$ cannot differ by~$1$. This is a contradiction.
	\end{proof}
	
	The following observation will be used in the proof of \fullcref{C1C2C3}{parity}:
	
	\begin{cor} \label{iota+tau}
		Let $\{a,b\}$ be a\/ $2$-element generating set of a finite abelian group~$G$, such that
		\begin{itemize}
			\item $\cayd(G; a,b)$ does not have a hamiltonian cycle,
			\item $|G|$ is even,
			and
			\item $|G : \langle a - b \rangle|$ is odd.
		\end{itemize}
		If $P$ and~$P'$ are two arc-disjoint hamiltonian paths in $\cayd(G; a,b)$, with initial vertices $\iota$ and~$\iota'$, and terminal vertices $\tau$ and~$\tau'$, then
		\[ \iota' + \tau' = \iota + \tau . \]
	\end{cor}
	
	\begin{proof}
		Since $\cayd(G; a,b)$ is not hamiltonian, there is no directed edge from~$\tau$ to~$\iota$, so we see from \cref{iota=tau} that 
		$\iota' \in \{\tau + a, \tau + b\}$
		and
		$\tau' \in \{\iota - a, \iota - b\}$,
		so
		\[ \iota' + \tau' \in 
		\{\iota + \tau, \ \iota + \tau \pm (a - b) \} . \]
		Therefore, since $a - b \notin 2G$ (because $|G|$ is even and $|G : \langle a - b \rangle|$ is odd), it will suffice to show $\iota' + \tau' \equiv \iota + \tau \pmod{2G}$.
		
		Write
		\[ \text{$\delta_b(P) = t \, o(a - b) + d$
			\ and \ 
			$\delta_b(P') = t' \, o(a - b) + d'$} , \]
		so we see from \fullcref{HPiff}{utk} (and \cref{SameNumCosetsAndD}) that $d = u_t(k)$ and $d' = u_{t'}(k')$ for some $k$ and~$k'$. We also see from \cref{MustBe<1} that
		\[ t + t' = |G : \langle a - b \rangle| - 1 . \]
		Since $|G : \langle a - b \rangle| - 1$ is even, we conclude that $t \equiv t' \pmod{2}$, so \fullcref{parity}{even} tells us that $d \equiv d' \pmod{2}$. 
		We know from \fullcref{arcforcing}{terminal-iota} that
		\begin{align*}
			\text{$\iota - \tau = d(a - b) + a$
				\ and \ 
				$\iota' - \tau' = d'(a - b) + a$}
			, \end{align*}
		so this implies
		\begin{align*}
			\iota' + \tau'
			&\equiv \iota' - \tau'
			&& \pmod{2G}
			\\&= d'(a - b) + a
			\\&\equiv d(a - b) + a
			&& \pmod{2G}
			\\&= \iota - \tau
			\\&\equiv \iota + \tau
			&& \pmod{2G}
			.  \qedhere \end{align*}
	\end{proof}


	\section{Cartesian product of more than two directed cycles} \label{C1xC2xC3Sect}
	
	\begin{cor} \label{HamPathsInC1xC2xC3xC4}
		If $C_1,C_2,\ldots,C_r$ are directed cycles \textup(of length $\ge 2$\textup), and $r \geq 4$, then the Cartesian product $C_1 \cartprod C_2 \cartprod \cdots \cartprod C_r$ has two arc-disjoint hamiltonian paths.
	\end{cor}
	
	\begin{proof}
		This argument is very easy (and classical: see the proof of \cite[Lem.~9.2]{Curran}, for example).
		\Cref{HamCycInC1xC2xC3} provides a hamiltonian cycle~$C$ in the Cartesian product $C_1 \cartprod C_2 \cartprod \cdots \cartprod C_{r-1}$. Then $C_1 \cartprod C_2 \cartprod \cdots \cartprod C_r$ has a spanning subdigraph that is isomorphic to $C \cartprod C_r$, and \cref{main} provides two arc-disjoint hamiltonian paths in this spanning subdigraph. Hamiltonian paths in a spanning subdigraph are also hamiltonian paths in the entire digraph.
	\end{proof}
	
	We do not know whether the Cartesian product of three directed cycles always has two arc-disjoint hamiltonian paths, but we can prove that the paths exist in most cases:
	
	\begin{prop} \label{C1C2C3}
		Assume $C_1$, $C_2$, and~$C_3$ are directed cycles \textup(of length~$\ge 2$\textup). If either
		\begin{enumerate}
			\item \label{C1C2C3-ham}
			the Cartesian product of two of the directed cycles has a hamiltonian cycle,
			or
			\item \label{C1C2C3-parity}
			the lengths of the three directed cycles do not all have the same parity \textup(i.e., if there is a directed cycle of even length and a directed cycle of odd length\textup),
			or
			\item \label{C1C2C3-len2}
			at least one of the directed cycles has length~$2$,
		\end{enumerate}
		then $C_1 \cartprod C_2 \cartprod C_3$ has two arc-disjoint hamiltonian paths.
	\end{prop}
	
	\begin{proof}
		(\ref{C1C2C3-ham}) This uses the same simple argument as the proof of \cref{HamPathsInC1xC2xC3xC4}. Assume, without loss of generality, that $C_1 \cartprod C_2$ has a hamiltonian cycle~$C$. Then $C \cartprod C_3$ is isomorphic to a spanning subdigraph of $C_1 \cartprod C_2 \cartprod C_3$, and \cref{main} provides two arc-disjoint hamiltonian paths in this spanning subdigraph.
		
		(\ref{C1C2C3-parity}) Write 
		\[ C_1 \cartprod C_2 \cartprod C_3 = \cayd(\Z_m \times \Z_n \times \Z_r; e_1, e_2, e_3) , \]
		where $\{e_1, e_2, e_3\}$ is the standard generating set of $\Z_m \times \Z_n \times \Z_r$.
		For $0 \le i < r$, let $X_i$ be the subdigraph that is induced by $\Z_m \times \Z_n \times \{i\}$, so $X_i \cong \cayd(\Z_m \times \Z_n; e_1, e_2)$.
		
		By \cref{main}, we may let $P_0$ and~$P_0'$ be two arc-disjoint hamiltonian paths in~$X_0$. Let $\iota_0$ and~$\iota_0'$ be their initial vertices, and let $\tau_0$ and~$\tau_0'$ be their terminal vertices.
		By assumption, we may assume that $m$ and $n$ have opposite parity (by permuting the factors). We may also assume $X_0$ does not have a directed hamiltonian cycle, for otherwise (\ref{C1C2C3-ham}) applies. Then we see from \cref{iota+tau} that
		\[ \tau - \iota' = \tau' - \iota . \]
		Therefore, if we let $\Delta \coloneqq \tau - \iota'$, then     \[ \text{$\tau = \iota' + \Delta$
			\ and \ 
			$\tau' = \iota + \Delta$} . \]
		For $0 \le i < r$, let $P_i$ and~$P_i'$ be the translates of~$P_0$ and~$P_0'$ by $(\Delta i, i)$, so $P_i$ and~$P_i'$ are arc-disjoint hamiltonian paths in~$X_i$.
		
		Let $(\iota_i, i)$ and~$(\iota_i', i)$ be the initial vertices of $P_i$ and~$P_i'$ and let $(\tau_i, i)$ and~$(\tau_i', i)$ be their terminal vertices. Then 
		\begin{align*}
			(\tau_i, i) + e_3
			&= (\tau + \Delta i, i + 1)
			= \bigl( (\iota' + \Delta) + \Delta i, i + 1 \bigr))
			\\&= \bigl( \iota' + \Delta (i + 1), i + 1 \bigr) 
			= (\iota_{i+1}', i + 1)
		\end{align*} 
		is the initial vertex of~$P_i'$.  This means there is a directed $e_3$-edge from the terminal vertex of~$P_i$ to the initial vertex of~$P_{i+1}'$. 
		Similarly, since $\tau' = \iota + \Delta$, we see that there is a directed $e_3$-edge from the terminal vertex of~$P_i'$ to the initial vertex of~$P_{i+1}$.
		
		Therefore, if we start with the union of all of the paths $P_i$ for $i$~even and $P_j'$ for $j$~odd, then we can add certain $e_3$-edges to construct a hamiltonian path~$P$ in $C_1 \cartprod C_2 \cartprod C_3$. Similarly, we can construct a hamiltonian path~$P'$ by adding appropriate  $e_3$-edges to the union of all of the paths $P_i$ for $i$~even and $P_j'$ for $j$~odd. (In the terminology of \cite[Defn.~9.4]{Curran}, $P$ and~$P'$ are ``laminated'' hamiltonian paths.) It is clear from the construction that these hamiltonian paths are arc-disjoint.
		
		(\ref{C1C2C3-len2}) Assume, without loss of generality, that $C_1$ has length~$2$. If $C_2$ has even length, then (since $C_1$ has length~$2$) it is easy to see (and well known) that $C_1 \cartprod C_2$ has a hamiltonian cycle, so (\ref{C1C2C3-ham}) applies. If $C_2$ has odd length, then (\ref{C1C2C3-parity}) applies.
	\end{proof}


	\section{\texorpdfstring{$2$}{2}-generated Cayley digraphs on finite abelian groups} \label{2genFinite}
	
	\begin{prop} \label{EvenIndex}
		\Cref{ArcDisjointCayleyConj} is true in all cases where the arc-forcing subgroup $\langle a - b \rangle$ has even index in~$G$.
	\end{prop}
	
	\begin{proof}
		Let $t = |G : \langle a - b \rangle|/2$, and let $d = u_t(1)$, so $H_t(d)$ is a hamiltonian path. By \cref{ut+ut}, we have 
		\[u_{t-1}(f_{t-1} - 1) = o(a - b) - 1 - d ,\]
		so $H_{t-1} \bigl( o(a - b) - 1 - d \bigr)$ is also a hamiltonian path.
		Furthermore,
		\begin{align*}
			\delta_b \bigl( H_t(d) \bigr) 
			& + \delta_b \bigl( H_{t-1} \bigl( o(a - b) - 1 - d \bigr) \bigr)
			\\&= \bigl[ t \, o(a - b) + d \bigr]
			+ \bigl[ (t-1) \, o(a - b) + \bigl( o(a - b) - 1 - d \bigr) \bigr]
			\\&= 2t \, o(a - b) - 1
			\\&= |G| - 1
			.\end{align*}
		Therefore, \fullcref{iflessthanone}{bb} provides two arc-disjoint hamiltonian paths.
	\end{proof}
	
	The open cases of \cref{ArcDisjointCayleyConj} reduce to two types of examples (and $G = \Z_n$ is cyclic in both types):
	
	\begin{prop} \label{ReduceToCyclic}
		It suffices to prove \cref{ArcDisjointCayleyConj} for the following two families of Cayley digraphs:
		\begin{enumerate}
			\item \label{ReduceToCyclic-pos}
			$\cayd(\Z_k; a, a + 1)$, where $k,a \in \Z^+$ \textup(and neither $a$ nor $a + 1$ is divisible by~$k$\textup), 
			and
			\item \label{ReduceToCyclic-neg}
			$\cayd(\Z_k; -a, a + 1)$, where $k,a \in \Z^+$, and $k$ is divisible by $2a + 1$ \textup(and neither $a$ nor $a + 1$ is divisible by~$k$\textup). 
		\end{enumerate}
	\end{prop}
	
	\begin{proof}
		Assume that $\{a,b\}$ is a $2$-element generating set of a finite abelian group~$G$. For convenience, let $\mathcal F = \langle a - b \rangle$. We may assume that $|G : \mathcal F|$ is odd, for otherwise \cref{EvenIndex} applies. Write $|G : \mathcal F| = 2t + 1$ and let $c = ta + tb$. 
		
		\refstepcounter{startcases}
		\begin{case}
			Assume that $c + a \neq 0$ and $c + b \neq 0$.
		\end{case}
		Let $a' = c + a$ and $b' = c + b$. Since $a',b' \in \langle a - b \rangle$, and $a' - b' = a - b$, we have 
		\[ \cayd( \langle a - b \rangle; a', b') \cong \cayd(\Z_k; \ell, \ell + 1) , \]
		where $k = o(a - b)$ and $b' = \ell(a - b)$. This Cayley digraph is listed in~(\ref{ReduceToCyclic-pos}), so we may assume that it has two arc-disjoint hamiltonian paths. 
		Thus, there exist $d$ and~$d'$, such that $H_0(d)$ and $H_0(d')$ are hamiltonian paths in $\cayd( \mathcal F; c + a, c + b)$, and $d + d' = o(a - b) - 1 + \epsilon$, where $|\epsilon| \le 1$. Then $H_t(d)$ and $H_t(d')$ are hamiltonian paths in $\cayd(G; a,b)$ (by the ``skewed generators argument\rlap,'' cf.\ \cite[Lem.~11]{morris2012nilp}). 
		
		Since 
		\begin{align*}
			\delta_b \bigl( H_t(d) \bigr) + \delta_b \bigl( H_t(d') \bigr)
			&= \bigl( t \, o(a - b) + d \bigr) + \bigl( t \, o(a - b) + d' \bigr)
			\\&= 2t \, o(a - b) + (d + d')
			\\&= \bigl( |G : \mathcal F| - 1 \bigr) |\mathcal F| + \bigl( o(a - b) - 1 + \epsilon \bigr)
			\\&= |G| - 1 + \epsilon
			, \end{align*}
		we conclude from \fullcref{iflessthanone}{bb} that $\cayd(G; a,b)$ has two arc-disjoint hamiltonian paths.
		
		\begin{case}
			Assume that either $c + a = 0$ or $c + b = 0$.
		\end{case}
		Assume without loss of generality that $c + a = 0$.
		For convenience, let $w = a - b$ and $k = |G|$. We have
		\[ 0 = c + a = (ta + tb) + a = (t+1)a + tb = (2t+1)b + (t+1)w , \]
		so $(2t + 1)b = -(t + 1)w$. Also note that 
		\[ o(w) = o(a - b) = \frac{k}{2t + 1} \]
		(so $k$ is divisible by $2t + 1$). Hence, as an abelian group, $G$~has the presentation
		\[ G = \left\langle b, w \mathrel{\Big|} (2t + 1)b = -(t + 1)w, \frac{k}{2t+1} w = 0 \right\rangle , \]
		so $G$, $b$, and~$w$ are uniquely determined (up to isomorphism) by $k$ and~$t$, and the assumptions that $|G| = k$, $ta + tb + a = 0$, and $w = a - b$ has order $k/(2t+1)$. Since $a = b + w$, it is also uniquely determined.
		
		On the other hand, it is clear that letting $G = \Z_k$, $a = 1 + t$, $b = -t$, and $w = 2t + 1$ provides an example (for any $k$ and~$t$, such that $k$ is divisible by $2t + 1$). Hence, we conclude that $\cayd(G; a,b)$ is isomorphic to the Cayley digraph that is listed in~(\ref{ReduceToCyclic-neg}).
	\end{proof}
	
	It is well known \cite[Thm.~459, p.~541]{HardyWright} that the probability that two random integers are relatively prime is $6/\pi^2 = 0.6079\cdots$. This has the following elementary consequence:
	
	\begin{lem}[{\cite[Thm.~8.5]{Curran}}] \label{TriangleAsympCount}
		If\/ $\mathsf N(m,n,e)$ is the number of primitive lattice points in the interior of the triangle $T(m,n,e)$, then
		\[ 
		\pushQED{\qed}
		\lim_{m,n \to \infty} \frac{\mathsf  N(m,n,e)}{\Area \bigl( T(m,n,e) \bigr)} = \frac{6}{\pi^2} . 
		\qedhere
		\popQED \]
	\end{lem}
	
	\begin{cor} \label{manyHP}
		Let $N(G; a, b)$ be the number of \textup(unordered\textup) pairs $\{P, P'\}$ of arc-disjoint hamiltonian paths in $\cayd(G; a,b)$. If $m = o(a)$ and $n = |G|/o(a)$ are sufficiently large, then
		\[ N(G; a, b) > \frac{|G|^2}{10} . \]
	\end{cor}
	
	\begin{proof}
		Let
		\begin{align*}
			R 
			&= \left\{ t \, o(a - b) + u_t(k) 
			\mathrel{\Big|}
			\begin{matrix}
				0 \le t < |G : \langle a - b \rangle|,
				\quad
				1 \le k < f_t, \\
				\text{$A_t(k)$ is in the interior of $T(m,n,e)$}
			\end{matrix} \right\}
			\\&\subseteq \{1,2, \ldots, |G| - 2 \}
			. \end{align*}
		Also let
		\[ \text{$R^+ = R \cup (R + 1)$
			and
			$R^- = R \cup (R - 1)$} . \]
		
		We claim that $|r - r'| \ge 2$ for all distinct $r,r' \in R$. To see this, first recall that 
		\[ u_t(k) - u_t(k - 1) = 2 h_t(k) \ge 2 . \]
		Second, note that, by \cref{ut+ut}, we have
		\begin{align*}
			\bigl( t \, o(a - b) + u_t(1) \bigr) - \bigl( (t-1) \, o(a - b)  + u_{t-1}(f_{t-1} - 1) \bigr)
			&= 2u_t(1) + 1
			. \end{align*} 
		If $A_t(1)$ is in the interior of $T(m,n,e)$, then $h_t(1) \ge 1$. Since $u_t(1) = h_t(1)$, this implies $2u_t(1) + 1 > 2$, which completes the proof of the claim.
		
		From the claim, we see that
		\[\# R^+ 
		= \# R^- 
		= 2 \cdot \#R 
		\approx 2 \cdot \frac{6}{\pi^2} \Area \bigl( T(m,n,e) \bigr)
		= \frac{6}{\pi^2} |G|
		. \]
		Therefore
		\begin{align*}
			\# \bigl( R^+ \cap (|G| - 1 - R^-) \bigr)
			&\ge \# R^+ + \# R^- - |G|
			\\&\approx 2 \cdot \frac{6}{\pi^2} |G| - |G|
			\\&> 0.2 |G|
			. \end{align*}
		For each $r \in R$, we know from the definition of~$R$ that there is a hamiltonian path $P$, such that $\delta_b(P) = r$. Hence, for each $r$ in the above intersection, we have hamiltonian paths $P$ and~$P'$, such that 
		\begin{align*}
			\delta_b(P) &\in \{r, r-1\} \\
			\intertext{and}
			\delta_b(P') &\in 
			\{|G| - 1 - r, |G| - r \}
			. \end{align*}
		Therefore 
		\[ \delta_b(P) + \delta_b(P')
		\in \{|G|, |G| - 1, |G| - 2\}| , \]
		so we see from (the proof of) \fullcref{iflessthanone}{bb} that these hamiltonian paths can be made arc-disjoint.
		
		This provides more than $0.2 |G|$ ordered pairs of arc-disjoint hamiltonian paths. So the number of unordered pairs is more than $0.1 |G|$. Furthermore, if $\{P_1, P_1'\}$ and $\{P_2, P_2'\}$ are two such pairs, then we know from the construction that $\delta_b(P_1) \notin \{\delta_b(P_2), \delta_b(P_2')\}$, so $\{P_1, P_1'\}$ is not a translate of $\{P_2, P_2'\}$. Therefore, since each pair has $|G|$ translates, the total number of unordered pairs of arc-disjoint hamiltonian paths is more than $0.1 |G|^2$.
		
		(Actually, there is a slight technical issue that the translates of $\{P, P'\}$ might not all be distinct. However, this can only happen if $P'$ is a translate of~$P$, which implies that $\delta_b(P) = \delta_b(P')$. Since $P$ and~$P'$ are arc-disjoint, this determines $\delta_b(P)$ up to an error of at most~$1$. So this problem is avoided by deleting a small number of values of~$r$ that are near $|G|/2$.)
	\end{proof}
	
	The following variant of \cref{TriangleAsympCount} provides an explicit lower bound on $m$ and~$n$ when $6/\pi^2$ is replaced with the smaller constant $1/2$:
	
	\begin{lem}[cf.\ {\cite[Thm.~8.5]{Curran}}] \label{TriangleCount}
		Let\/ $\mathsf N(m,n,e)$ be the number of primitive lattice points in the interior of the triangle $T(m,n,e)$.
		If $m,n \ge 300$, then 
		\[ \mathsf N(m,n,e) > \frac{1}{2} \Area \bigl( T(m,n,e) \bigr) . \]
	\end{lem}
	
	\begin{proof}[Sketch of proof \textup{(cf.\ \cite[proof of Thm.~8.5]{Curran})}]
		For every triangle~$T$, let $N(T)$ (resp.\ $P(T)$) be the number of lattice points (resp.\ primitive lattice points) that are in the interior of~$T$ and are not on the $y$-axis.
		
		\refstepcounter{startcases}
		\begin{case} \label{TriangleCountPf-smalle}
			Assume $-n \le e \le n$.
		\end{case}
		Then $\min \bigl( |x|, |y| \bigr) \le \min(m,n)$ for all $(x,y) \in T(m,n,e)$, so we have $N \bigl( \frac{1}{k} T \bigr) = 0$ for all $k > \min(m,n)$. Also, it is elementary to see that
		\[ \left| N \! \left( \frac{1}{k} \, T \right) - \Area \! \left( \frac{1}{k} \, T \right) \right| \le \frac{2 \, \max(m,n)}{k} \]
		for all $k \in \Z^+$.
		Therefore, letting $A = \Area(T)$ and $\mathsf{min} = \min(m,n)$, we have
		\begin{align*}
			P(T)
			&= \sum_{k=1}^{\mathsf{min}} \mu(k) \, N \! \left( \frac{1}{k} \, T \right)
			\\&= \sum_{k=1}^{\mathsf{min}} \mu(k) \left(  \frac{A}{k^2} \pm \frac{2 \, \max(m,n)}{k} \right)
			\\&= A \left( \frac{6}{\pi^2} - \sum_{k > \mathsf{min}} \frac{\mu(k)}{k^2} \right)
			\pm 2 \, \max(m,n) (1 + \log \mathsf{min})
			\\&> A \left( \frac{6}{\pi^2} - \frac{1}{\mathsf{min}} \right) - 2 \max(m,n) (1 + \log \mathsf{min})
			\\&= A \left( \frac{6}{\pi^2} - \frac{1}{\mathsf{min}} - \frac{4 (1 + \log \mathsf{min})}{\mathsf{min}} \right)
			. \end{align*}
		Since $\mathsf{min} \ge 300$ (in fact, $\mathsf{min} \ge 252$ would suffice), we have 
		\[ \frac{1}{\mathsf{min}} + \frac{4 (1 + \log \mathsf{min})}{\mathsf{min}} < 0.1079, \]
		so we conclude that $P(T) > 0.5 A$.
		
		\begin{case} \label{TriangleCountPf-general}
			The general case.
		\end{case}
		By applying an appropriate element $\begin{pmatrix} 1 & * \\ 0 & 1 \end{pmatrix}$ of $\mathrm{SL}(2,\Z)$, we may assume $n - m \le e \le n$. Since we may assume that \cref{TriangleCountPf-smalle} does not apply, we then have $n - m \le e < -n$. The $y$-axis divides $T(m,n,e)$ into two smaller triangles $T_1$ and~$T_1'$. Letting $p = mn/(n - e) \ge n$, we have
		\begin{itemize}
			\item the vertices of $T_1$ are $(0,0)$, $(n,0)$, and  $(0,p)$, 
			and 
			\item the vertices of $T_1'$ are $(0,0)$, $(0,p)$, and $(e,m)$. 
		\end{itemize}
		Since $p \ge n$, we see that \cref{TriangleCountPf-smalle} applies to the triangle~$T_1$, so
		\[ P(T_1) > \frac{1}{2} \Area(T_1) . \]
		
		Now, rotating $T_1'$ by $90^\circ$ clockwise around the origin yields the triangle $T_1''$ with vertices $(0,0)$, $(p,0)$ and $(m, -e)$.  Since $p > n$ and $-e > n$, this triangle satisfies the hypotheses of the proposition. Therefore, either \cref{TriangleCountPf-smalle} applies to $T_1''$, or the argument in the preceding paragraph divides $T_1''$ into two triangles $T_2$ and $T_2'$, such that \cref{TriangleCountPf-smalle} applies to $T_2$, and $T_2'$ satisfies the hypotheses of the proposition (after applying an appropriate element of $\mathrm{SL}(2,\Z)$). Continuing in this way, we see that $T(m,n,e)$ can be decomposed into a finite union
		\[ T = T_1 \cup T_2 \cup \cdots \cup T_\ell \]
		such that \cref{TriangleCountPf-smalle} applies to each~$T_i$ (after applying an appropriate element of $\mathrm{SL}(2,\Z)$). Therefore, we have $P(T_i) > \frac{1}{2} \Area(T_i)$ for each~$i$, so
		\[ P(T) 
		\ge \sum_i P(T_i)
		> \sum_i \frac{1}{2} \Area(T_i)
		= \frac{1}{2} \Area(T)
		. \qedhere \]
	\end{proof}
	
	\begin{prop} \label{LargeIndex}
		\Cref{ArcDisjointCayleyConj} is true in all cases where we have either $|G : \langle a \rangle| \ge 600$ or $|G : \langle b \rangle| \ge 600$.
	\end{prop}
	
	\begin{proof}
		Assume, without loss of generality, that $|G : \langle a \rangle| \ge 600$.
		We consider two cases.
		
		\refstepcounter{startcases}
		\begin{case}
			Let $m \ge 300$.
		\end{case}
		Then one can see that by \cref{TriangleCount} we have $2 \, \mathsf N(m,n,e) > |G| / 2$, so the argument in the proof of \cref{manyHP} establishes that the number of pairs of arc-disjoint hamiltonian paths is greater than~$0$.
		
		\begin{case}
			Assume $m < 300$.
		\end{case}
		Since $e < m$ and $n = |G : \langle a \rangle| \ge 600$, this implies $n > m + e$. Hence, we see from \cref{dgap} that 
		for all~$k$, such that $n - 1 \le k \le m(n-1)$, there is a hamiltonian path~$P$ in $\cayd(G; a,b)$, such that 
		\[ k \le \delta_b(P) \le k + 2 \left\lfloor \frac{n - 1}{2} \right\rfloor. \]
		The desired conclusion can now be obtained by combining this with \cref{string}, as in the proof of \cref{main}.
	\end{proof}


	\section{\texorpdfstring{$2$}{2}-generated Cayley digraphs on infinite abelian groups} \label{2genInfinite}
	
	In this \lcnamecref{2genInfinite}, we study the natural analogue of \cref{arcforcing} for infinite digraphs. (Note that the hamiltonicity or hamiltonian decomposability of infinite Cayley (di)graphs has been studied by several authors, see \cite{BryantEtAl, ErdeLehner, ErdeLehnerPitz, jungreis1985infinite, circleabe, witte1984survey}.)
	The foundation of our results is the observation that the basic theory of the arc-forcing subgroup easily generalizes to the infinite case:
	
	\begin{lem}[cf.\ \fullcref{arcforcing}{nonterminal}] \label{coset_travel}
		Let $\{a,b\}$ be a $2$-element subset of a group~$G$ \textup(which may be infinite\textup).  If $P$ is a spanning subdigraph of $\cayd(G;a,b)$, such that the indegree and outdegree of every vertex is~$1$, then every right coset of $\langle a - b \rangle$ either travels by $a$ or travels by~$b$.\qed
	\end{lem}
	
	This has the following easy consequence, which is the infinite analogue of \fullcref{HCiff}{Rankin}:
	
	\begin{prop} \label{InfiniteHam}
		Let $\{a,b\}$ be a $2$-element generating set of an abelian group~$G$ \textup(such that $a \neq b$\textup), and let $I = |G : \langle a - b \rangle|$. 
		\begin{enumerate} \setcounter{enumi}{-1}
			\item \label{InfiniteHam-index}
			If $\cayd(G; a,b)$ has a two-way infinite hamiltonian path, then $I < \infty$.
			\item \label{InfiniteHam-assumeHP}
			Suppose $P$ is a two-way infinite hamiltonian path in $\cayd(G; a, b)$. If $k$ is the number of cosets of $\langle a - b \rangle$ that travel by~$a$, and $\ell$ is the number of cosets of $\langle a - b \rangle$ that travel by~$b$, then $k + \ell = I$, and $\langle k a + \ell b \rangle = \langle a - b \rangle$. 
			\item \label{InfiniteHam-constructHP}
			Conversely, suppose $k + \ell = I < \infty$ and $\langle k a + \ell b \rangle = \langle a - b \rangle$. If $P$ is any spanning subdigraph of $\cayd(G; a, b)$, such that the outdegree of every vertex is~$1$, and exactly $k$~cosets of $\langle a - b \rangle$ travel by~$a$ and exactly $\ell$~cosets of $\langle a - b \rangle$ travel by~$b$, then $P$ is a two-way infinite hamiltonian path in $\cayd(G; a, b)$.
		\end{enumerate}
	\end{prop}
	
	\begin{proof}[Proof \upshape(cf.\ {\cite[Thm.~4]{Rankin})}]
		(\ref{InfiniteHam-index})
		Let $P = \ldots, v_{-2}, v_{-1}, v_0, v_1, v_2, \ldots$ be a two-way infinite hamiltonian path in $\cayd(G; a,b)$. Assume, without loss of generality, that $v_0 = 0$. Then, since $a \equiv b \pmod{\langle a - b \rangle}$, we have $v_i \in ia + \langle a - b \rangle$ for every $i \in \Z$. If $I = \infty$, this implies that each $v_i$ is in a different coset of $\langle a - b \rangle$. Since $\{v_i\}$ is a list of all of the elements of~$G$, we conclude that $v_i$ is the only element of its coset, so each coset has only one element. This means $|\langle a - b \rangle| = 1$, which contradicts the fact that $a \neq b$.
		
		(\ref{InfiniteHam-assumeHP})
		Since every coset of $\langle a - b \rangle$ travels by either $a$ or~$b$ (but not both), we have $k + \ell = I$. Now, write $P = \ldots, v_{-2}, v_{-1}, v_0, v_1, v_2, \ldots$, and assume, without loss of generality, that $v_0 = 0$. Then 
		\[ v_i \in \langle a - b \rangle \iff \text{$i$ is divisible by~$I$} \]
		and
		\[ \text{$v_{jI} = j(ka + \ell b)$ for every $j \in \Z$} . \]
		Since $\{v_i\}_{i \in \Z}$ is a list of all of the elements of~$G$, this implies that $\{j(ka + \ell b)\}$ is a list of all of the elements of $\langle a - b \rangle$, which means that $ka + \ell b$ generates $\langle a - b \rangle$.
		
		(\ref{InfiniteHam-constructHP})
		Let $\ldots, v_{-2}, v_{-1}, v_0, v_1, v_2, \ldots$ be the path component of~$P$ that contains~$0$ (with $v_0 = 0$). We wish to show that every vertex is in this component.
		
		Suppose $v \in G$. Since
		\[ \langle a, k a + \ell b \rangle
		= \langle a, a - b \rangle
		= \langle a, b \rangle, \]
		we have $v \in ia + \langle k a + \ell b \rangle$ for some $i \in \Z$. Since $v_i \equiv ia \pmod{a - b}$, this implies there exists $j \in \Z$, such that 
		\[ v = v_i + j(k a + \ell b) = v_{i + jI} . \qedhere \]
	\end{proof}
	
	The following observation could easily be proved directly, but we present it as a simple application of \cref{InfiniteHam}.
	
	\begin{cor} \label{ZxZm}
		If $\{e_1, e_2\}$ is the standard generating set of $\Z \times \Z_m$ \textup(and $m \ge 2$\textup), then the Cayley digraph\/ $\cayd(\Z \times \Z_m;e_1,e_2)$ has a unique two-way infinite hamiltonian path, up to translations.
	\end{cor}
	
	\begin{proof}
		(existence) 
		Let $k = 1$ and $\ell = m - 1$. Then 
		\[ k + \ell = m = | \, \Z \times \Z_m : \langle e_1 - e_2 \rangle \,| \]
		and 
		\[ ke_1 + \ell e_2 = e_1 + (m - 1) e_2 = e_1 - e_2 , \]
		so \fullcref{InfiniteHam}{constructHP} tells us that $\cayd(\Z \times \Z_m;e_1,e_2)$ has a two-way infinite hamiltonian path.
		
		(uniqueness)
		Let $k$ and~$\ell$ be as in \fullcref{InfiniteHam}{assumeHP}. Then
		\[ \langle (k, \ell) \rangle
		= \langle k e_1 + \ell e_2 \rangle
		= \langle e_1 - e_2 \rangle
		= \langle (1, -1) \rangle
		, \]
		so the projection of this subgroup to~$\Z$ is surjective. We conclude that $k = 1$. This means that precisely one coset of $\langle a - b \rangle$ travels by~$e_1$ (and all other cosets travel by~$e_2$). Therefore, a two-way infinite hamiltonian path is determined by choosing which coset of $\langle e_1 - e_2 \rangle$ travels by~$e_1$. Since cosets are translates of each other, this implies that all two-way infinite hamiltonian paths are translates of each other.
	\end{proof}
	
	\begin{cor} \label{InfiniteADgeneral}
		Let $\{a,b\}$ be a $2$-element generating set of an abelian group~$G$ \textup(such that $a \neq b$\textup), and let $I = |G : \langle a - b \rangle|$. The digraph $\cayd(G; a,b)$ has two arc-disjoint two-way infinite hamiltonian paths if and only if $I < \infty$ and there exist $k, \ell \in \bigl\{ 0,1,\ldots, I \bigr\}$, such that $k + \ell = I$ and 
		\[ \langle a - b \rangle = \langle k a + \ell b \rangle = \langle \ell a + k b \rangle . \]
	\end{cor}
	
	\begin{proof}
		($\Rightarrow)$ Let $P$ be a two-way infinite hamiltonian path, and let $k$ and~$\ell$ be as in \fullcref{InfiniteHam}{assumeHP}, so $k + \ell = I$ and $\langle k a + \ell b \rangle = \langle a - b \rangle$. If $P'$ is a two-way infinite hamiltonian path that is arc-disjoint from~$P$, then the cosets that travel by~$a$ in~$P'$ are the cosets that travel by~$b$ in~$P$, so the number of cosets that travel by~$a$ in~$P'$ is~$\ell$, and the number of cosets that travel by~$b$ is~$k$. So \fullcref{InfiniteHam}{assumeHP} tells us that $\langle \ell a + k b \rangle = \langle a - b \rangle$.
		
		($\Leftarrow$) Choose $k$ cosets of $\langle a - b \rangle$. Let $P$ be the spanning subdigraph in which the outdegree of every vertex is~$1$, and these particular $k$~cosets of $\langle a - b \rangle$ travel by~$a$, and the other $\ell$~cosets travel by~$b$.  Then let $P'$ be the spanning subdigraph in which the outdegree of every vertex is~$1$, and the $k$ chosen cosets of $\langle a - b \rangle$ travel by~$b$, and the other $\ell$~cosets travel by~$a$. It is clear from the construction that $P'$ is arc-disjoint from~$P$ (since a vertex travels by~$a$ in~$P'$ if and only if it travels by~$b$ in~$P$). Furthermore, we see from \fullcref{InfiniteHam}{constructHP} that $P$ and~$P'$ are two-way infinite hamiltonian paths.
	\end{proof}
	
	This can be made much more explicit:
	
	\begin{prop} \label{infiniteAD}
		Assume $G$ is an infinite abelian group.
		\begin{enumerate}
			\item \label{infiniteAD-ab}
			There exist $a,b \in G$, such that $\cayd(G; a, b)$ has a two-way infinite hamiltonian path if and only if $G$ is isomorphic to either $\Z$ or $\Z \times \Z_m$, for some $m \ge 2$.
			\item \label{infiniteAD-Z}
			For $a,b \in \Z$, the Cayley digraph $\cayd(\Z; a,b)$ has two arc-disjoint two-way infinite hamiltonian paths if and only if $a$ and~$b$ are odd, and
			\[ \text{either $\{a,b\} = \{1, -1\}$ or $a + b = \pm 2$}. \]
			\item \label{infiniteAD-ZxZm}
			For $a,b \in \Z \times \Z_m$, with $m \ge 2$, the Cayley digraph $\cayd(\Z \times \Z_m; a,b)$ has two arc-disjoint two-way infinite hamiltonian paths if and only if either
			\begin{enumerate}
				\item \label{infiniteAD-ZxZm-xy}
				$\{a,b\} = \{ (1, x), (-1, y)\}$, for some $x,y \in \Z_m$, such that $\langle x + y \rangle = \Z_m$,
				or
				\item \label{infiniteAD-ZxZm-2}
				$m = 2$, $a = (0,1)$, and $b \in \{\pm 1\} \times \Z_2$, perhaps after interchanging $a$ and~$b$.
			\end{enumerate}
		\end{enumerate}
	\end{prop}
	
	\begin{proof}
		
		\medbreak
		
		(\ref{infiniteAD-ab} $\Leftarrow$) It is obvious that $\cayd(\Z; 1, -1)$ has a two-way infinite hamiltonian path. The remaining case is immediate from \cref{ZxZm}.
		
		\medbreak
		(\ref{infiniteAD-Z} $\Leftarrow$) If $a , b \in \{\pm 1\}$, then $\cayd(\Z; a)$ and $\cayd(\Z; b)$ are two arc-disjoint two-way infinite hamiltonian paths in $\cayd(\Z; a, b)$. 
		
		We may now assume $a + b = \pm 2$, and $a \neq b$. Since $a + b$ is even, we may write $a - b = 2\ell$, for some nonzero $\ell \in \Z$. Assume without loss of generality that $k > 0$ (by replacing $a$ and~$b$ with their negatives if necessary). Letting $k = \ell$, we have 
		\[ k + \ell = 2\ell = a - b = |\Z : \langle a - b \rangle| \]
		and
		\[ \langle k a + \ell b \rangle
		= \langle \ell a + k b \rangle
		= \langle \ell (a + b) \rangle
		= \langle 2 \ell \rangle
		= \langle a - b \rangle , \]
		so \cref{InfiniteADgeneral} tells us that $\cayd(G; a, b)$ has two arc-disjoint two-way infinite hamiltonian paths.
		
		\medbreak
		
		(\ref{infiniteAD-ZxZm-xy} $\Leftarrow$) Let $\overline{\phantom{x}} \colon \Z \times \Z_m \to \Z$ be the natural projection, and assume, without loss of generality, that $a = (1, x)$ and $b = (-1, y)$. Then $\overline {a - b} = 2$, so $|G \colon \langle a - b \rangle| = 2m$. Therefore, if we let $k = m + 1$ and $\ell = m - 1$, then
		\[ k + \ell 
		= 2m 
		= |G : \langle a - b \rangle | \]
		and
		\[ \overline{k a + \ell b}
		= k \overline a + \ell \overline b
		= (m + 1) \cdot 1 + (m - 1) \cdot (-1)
		= 2 
		, \]
		so 
		\[ |G : \langle k a + \ell b \rangle|
		= 2m
		= |G : \langle a - b \rangle| , \]
		so 
		\[ \langle k a + \ell b \rangle = \langle a - b \rangle . \]
		A similar calculation shows $\langle \ell a + k b \rangle = \langle a - b \rangle$. So \cref{InfiniteADgeneral} tells us that $\cayd(G; a, b)$ has two arc-disjoint two-way infinite hamiltonian paths.
		
		\medbreak
		
		(\ref{infiniteAD-ZxZm-2} $\Leftarrow$) 
		Since $a - b$ is of the form $(\pm 1, *)$, we have $I = 2$. Let $k = \ell = 1$. Then $k + \ell = I$, and, since $a = -a$, we have
		\[ \langle a - b \rangle
		= \langle a + b \rangle 
		= \langle k a + \ell b \rangle
		= \langle \ell a + k b \rangle . \]
		So \cref{InfiniteADgeneral} tells us that $\cayd(G; a, b)$ has two arc-disjoint two-way infinite hamiltonian paths.
		
		\medbreak
		
		(\ref{infiniteAD-ab} $\Rightarrow$) The structure theorem for finitely generated abelian groups \cite[4.2.10]{robinson2012course} tells us that 
		\[ G \cong \Z^r \times \Z_{n_1} \times \Z_{n_2} \times \cdots \times \Z_{n_s}, \]
		where $n_{i+1}$ is a multiple of $n_i$ for $1 \le i < s$ (and $n_i \ge 2$ for all~$i$).
		
		Since $G$ is infinite, we have $r \ge 1$. On the other hand, we know from \fullcref{InfiniteHam}{assumeHP} that $|G : \langle a - b \rangle| < \infty$, so $G$ has a cyclic subgroup of finite index, so $r \le 1$. We conclude that $r = 1$.
		
		The minimum cardinality of a generating set of~$G$ is $r + s$. Since $G = \langle a , b \rangle$, this implies $r + s \le 2$. Since $r = 1$, we conclude that $s \in \{0,1\}$. If $s = 0$, then $G \cong \Z$. If $s = 1$, then we may let $m = n_1$.
		
		\medbreak
		
		(\ref{infiniteAD-Z}~$\Rightarrow$, \ref{infiniteAD-ZxZm}~$\Rightarrow$)
		Let $G = \Z \times \Z_m$ for (where $m = 1$ if the group is~$\Z$), let $\overline{\phantom{x}} \colon \Z \times \Z_m \to \Z$ be the natural projection, and let $\overline I$ be the absolute value of $\overline a - \overline b$, so $\langle \overline{a} - \overline{b} \rangle = \langle \overline I \rangle$.
		
		If $\overline I = 0$, then $|G : \langle a - b \rangle| = \infty$, which contradicts the conclusion of \fullcref{InfiniteHam}{index}, so we must have $a = b$. Then $G = \langle a, b \rangle = \langle a \rangle \cong \Z$ and $a = b = \pm 1$, so $a + b = \pm 2$. This implies that $a$ and~$b$ have the same parity. They cannot both be even (since $\langle a, b \rangle = \Z$), so $a$ and~$b$ are odd. Thus, the situation is described in part~(\ref{infiniteAD-Z}) of the statement of the \lcnamecref{infiniteAD}.
		
		We may now assume $\overline I > 0$. 
		Then we claim that $|G : \langle a - b \rangle| = m \, \overline I$. 
		We have
		\[ | G : \langle a - b \rangle |
		= \left| \frac{G}{\langle a - b \rangle} \right|
		= \left| \frac{G}{\langle a - b \rangle + \Z_m} \right| \cdot \left| \frac{\langle a - b \rangle + \Z_m}{\langle a - b \rangle} \right|
		. \]
		Also note that
		\[ \frac{G}{\langle a - b \rangle + \Z_m}
		\cong \frac{G / \Z_m}{\bigl( \langle a - b \rangle + \Z_m \bigr) / \Z_m}
		\cong \frac{\Z}{\langle \overline a - \overline b \rangle}
		= \frac{\Z}{\langle \overline I \rangle} \]
		has order~$\overline I$, and
		\[ \frac{\langle a - b \rangle + \Z_m}{\langle a - b \rangle}
		\cong \frac{\Z_m}{\langle a - b \rangle \cap \Z_m}
		= \frac{\Z_m}{\{0\}}
		\cong \Z_m \]
		has order~$m$. 
		Therefore the claim is proved.
		By \cref{InfiniteADgeneral}, there exist $k, \ell \ge 0$, such that $k + \ell = m \, \overline I$, and 
		\[ \langle k a + \ell b \rangle
		= \langle \ell a + k b \rangle
		= \langle a - b \rangle . \]
		Therefore
		\[ \langle k \overline a + \ell \overline b \rangle
		= \langle \ell \overline a + k \overline b \rangle
		= \langle \overline a - \overline b \rangle 
		= \langle \overline I \rangle, \]
		so we may assume
		\[ \text{$k \overline a + \ell \overline b = \overline I$ and $\ell \overline a + k \overline b = \pm \overline I$} . \]
		Adding these two equations, we conclude that 
		\[ \text{$(k + \ell)(\overline a + \overline b)$ is either $0$ or~$2 \overline I$.} \]
		Since $k + \ell = m \overline I$, this implies
		\[ m(\overline a + \overline b) \in \{0, 2\} . \]
		
		\refstepcounter{startcases}
		\begin{case}
			Assume $m(\overline a + \overline b) = 0$.
		\end{case}
		This means $\overline a = - \overline b$. Since $\gcd(\overline a, \overline b) = 1$, this implies $\overline a = \pm 1$ and $\overline b = \pm 1$ (so $\overline I = 2$). If $G = \Z$, then the situation is described in part~(\ref{infiniteAD-Z}) of the statement of the \lcnamecref{infiniteAD}. 
		
		Therefore, we may assume $G \not\cong \Z$, so $m > 1$. Write $a = (1, x)$ and $b = (-1, y)$ (perhaps after interchanging $a$ and~$b$). Since 
		\[ G 
		= \langle a, b \rangle
		= \langle a, a + b \rangle
		= \langle (1, x), (0, x + y) \rangle
		,\]
		it is clear that $\langle x + y \rangle = \Z_m$. So we are in the situation specified by part~(\ref{infiniteAD-ZxZm-xy}) of the statement of the \lcnamecref{infiniteAD}. 
		
		\begin{case}
			Let $m(\overline a + \overline b) = 2$.
		\end{case}
		This immediately implies that $m \in \{1,2\}$. 
		
		If $m = 1$ (and $\overline a + \overline b = 2$), then $G = \Z$, and we have $a + b = \overline a + \overline b = 2$. (This implies that $a$ and~$b$ have the same parity. They cannot both be even, since $\langle a, b \rangle = G$, so $a$ and~$b$ must be odd.) So we are in a situation that is described in part~(\ref{infiniteAD-Z}) of the statement of the \lcnamecref{infiniteAD}. 
		
		Assume now that $m = 2$ (and $\overline I = 1$). Then $\overline a + \overline b = 1$. Note that $\langle 2 \overline a \rangle \neq \Z$ and $\langle 2 \overline b \rangle \neq \Z$. Since $\langle k \overline a + \ell \overline b \rangle = \Z$ and $k + \ell = m \, \overline I = 2$, we conclude that $k = \ell = 1$. This implies 
		\[ \overline a + \overline b = k \overline a + \ell \overline b  = \overline I = 1 . \]
		Since we also have $\overline a - \overline b = \pm \overline I = \pm 1$, we conclude that $\{\overline a, \overline b\} = \{0, 1\}$. So we are in the situation that is described in part~(\ref{infiniteAD-ZxZm-2}) of the statement of the \lcnamecref{infiniteAD}. 
	\end{proof}
	
	\begin{cor} \label{ZmxZ}
		$\cayd(\Z \times \Z_m; e_1, e_2)$ has two arc-disjoint two-way infinite hamiltonian paths if and only if $m = 2$.
	\end{cor}
	
	\begin{proof}
		($\Leftarrow$) Let $k = \ell = 1$. Then 
		\[ k + \ell = 2 = |\Z \times \Z_2 : \langle e_1 - e_2 \rangle| \]
		and
		\[ \langle ke_1 + \ell e_2 \rangle
		= \langle \ell e_1 + k e_2 \rangle
		= \langle e_1 + e_2 \rangle
		= \langle e_1 - e_2 \rangle
		, \]
		so \cref{InfiniteADgeneral} provides two arc-disjoint two-way infinite hamiltonian paths.
		
		($\Rightarrow$) Since $e_2 = (0,1)$ is obviously not of the form $(\pm1, *)$, it is clear that the generating set $\{e_1, e_2\}$ is not of the form specified in \fullcref{infiniteAD}{ZxZm-xy}, so it must be part~(\ref{infiniteAD-ZxZm-2}) of the \lcnamecref{infiniteAD} that applies. So $m = 2$.
	\end{proof}

	\bibliographystyle{DarijaniMiraftabMorris}
	\bibliography{references.bib}
	
\end{document}